\documentclass[11pt,reqno,a4paper]{article} 
\usepackage{authblk}
\usepackage{graphicx} 
\usepackage[top=2cm, bottom=2cm, left=2cm, right=2cm] {geometry}
\usepackage{amscd,amsmath,amstext,amsfonts,amsbsy,
amssymb,amsthm,mathrsfs,float,latexsym}
\usepackage{multicol,graphicx,array,multirow,color,lineno}
\usepackage{sidecap,url,fancybox,fancyhdr}
\usepackage[colorlinks=true,linkcolor=blue,citecolor=magenta]{hyperref}
\usepackage{comment}

\newcommand{\fu}{M_1}
\newcommand{\fo}{M_2}
\usepackage{subcaption}

\usepackage{enumitem}

\newcommand{\jeven}{\lfloor j \rfloor_{\mathsf{even}}}

\renewcommand{\epsilon}{\varepsilon}

\newcommand{\bra}[1]{\left(#1\right)}

\newtheorem{theorem}{{\bf Theorem}}[section]
\theoremstyle{definition} 
\theoremstyle{assumption} 

\theoremstyle{plain}
\newtheorem{lemma}{Lemma}[section]

\newtheorem{remark}{Remark}[section]

\allowdisplaybreaks

\title{\Large \bf Two-sided estimates of Lyapunov exponents  for Milstein schemes of non-autonomous stochastic differential equations}

\author{Vu Thi Hue$^{a}$\footnote{hue.vuthi@hust.edu.vn}, Quoc Bao Tang$^{b}$\footnote{quoc.tang@uni-graz.at, baotangquoc@gmail.com}, Do Duc Thuan$^{a}$\footnote{thuan.doduc@hust.edu.vn}}

\date{}

\begin{document}
\maketitle

\vspace*{-1cm}

\begin{center}
{\small
$^{a}$Faculty of Mathematics and Informatics, Hanoi University of Science and Technology, \\
1 Dai Co Viet, Bach Mai, Hanoi, Vietnam  \vspace{0.2cm}\\ 
$^{b}$Department of Mathematics and Scientific Computing, University of Graz, \\ Heinrichstrasse 36, 8010 Graz, Austria \vspace{0.2cm}
}
\end{center}

\begin{abstract} 
The stabilising effect of multiplicative noise for stochastic differential equations, though counterintuitive, has been observed and investigated extensively in last decades. In practice, it is desirable to know if such stabilisation holds also for the discretised setting. In this paper, we address this problem by means of sharp upper and lower estimates for Lyapunov exponents of Milstein schemes for non-autonomous stochastic differential equations. These estimates provide precise large time behaviour in both almost sure and $p$-moment sense. In particular, our results show the preservation of stabilisation from the continuum setting to the discretised setting. One main idea of our analysis is to exploit the second order term concerning the stochastic noise from the Milstein scheme to obtain precise estimates for Taylor expansions of logarithmic and power functions.
\medskip

    
\end{abstract}

\tableofcontents

\section{Introduction}

\subsection{Problem formulation and state-of-the-art}

Stabilisation is one of the central questions in the study of dynamical systems, especially concerning applications in engineering. The remarkable stabilising effect of stochastic noise has been already observed in \cite{kapica1951dynamical}, and rigorously justified first in \cite{khas1969stability} and extended to a linear stochastic differential systems \cite{arnold1983stabilization}. Since then, stochastic stabilisation, and also destabilisation, has become an active topic, where one of the main focuses concerns  Lyapunov exponents, which is well-known to describe the large time behaviour of dynamical systems,  see e.g. \cite{kushner1968stability,meerkov1982condition,bellman1985stability,mao1994stochastic,caraballo2003stochastic,appleby2008stabilization,braverman2020global} references therein.   
 
\medskip  

In practical problems, one usually works with approximate solutions via suitable numerical schemes, where the fundamental issues are the converging from the approximate solutions to the exact one and the convergence rate. However, even when this convergence is established, many important properties of the exact solution are not guaranteed to be preserved at the discretised level, such as the positivity and stability of solutions 
\cite{mora2017stable,buckwar2010towards,fang2023strong}. This gives rise to the study of numerical schemes preserving the exact solution's important physical properties, which has become an active research direction of the numerical analysis in general and of stochastic differential equations (SDEs) in particular.  

\medskip
 
Among the important properties of solutions to SDEs, the preservation of the almost sure exponential stability and instability for numerical schemes of the underlying SDE are of importance and interest, which are the main concerns of this work. 

\medskip

Let us begin with the linear non-autonomous SDE of finding $X:\mathbb{R}_+ \times \Omega \to \mathbb{R}$ such that   
\begin{equation}
    dX(t) = a(t)X(t)dt + \alpha X(t)dW(t)  \quad \text{for } t>0, \qquad 
    X(0) = x_0, 
\label{SDE:Continuous0}
\end{equation}
where $W$ is a standard Brownian motion on a filtered probability space $(\Omega,\mathcal F, \mathbb{F}, \mathbb{P})$, $x_0 \not = 0$ is the initial data, $\alpha$ is a real constant, 
and $a(t)$ is a function of time.
Using the It\^o formula,  
the solution to \eqref{SDE:Continuous0} is given by 
\begin{align*}
  X(t) &= x_0  \exp \left ( \int_0^t \left(a(s)-\dfrac{\alpha^2}{2}\right) ds + \alpha \int_0^t dW(s) \right ) .  
\end{align*}
An application of the law of the iterated logarithm with $ \limsup_{t\to \infty} \frac{|W(t)|}{t} = 0$ a.s. implies an estimate for the Lyapunov exponent of \eqref{SDE:Continuous0} as 
\begin{align}\label{Lyapunov_exp}
    \limsup_{t\to \infty} \frac{1}{t}\log |X(t)|  =  \left( \limsup_{t\to \infty} \frac{1}{t} \int_0^t  a(s) ds \right) - \frac{1}{2} \alpha^2  \quad \text{almost surely},
\end{align}
provided 
\begin{equation}\label{finiteness_a}
	\limsup_{t\to\infty}\frac{1}{t}\int_0^{t}a(s)ds \quad \text{is finite}.
\end{equation}
Therefore, the zero solution $X=0$ is \textit{almost sure exponentially stable} whenever $(a,\alpha)\in \mathbb{V}_{\mathsf{nonaut}}$ where $$ \mathbb{V}_{\mathsf{nonaut}}:=\left\{(a, \alpha) \in L^1_{\mathsf{loc}}(0,\infty) \times \mathbb{R}: \left(\limsup_{t\to \infty} \frac{1}{t} \int_0^t  a(s) ds \right) -\frac{1}{2} \alpha^2<0 \right\}.$$  
Particularly in the autonomous case, i.e. $a(t) \equiv a \in \mathbb R$,
the stable region is reduced to the well known stability region in $\mathbb R^2$
$$ \mathbb{V}_{\mathsf{aut}}:=\left\{(a, \alpha) \in \mathbb{R}^2: a -\frac{1}{2} \alpha^2<0 \right\}.$$
Now it is obvious when $0<\limsup_{t\to \infty} \frac{1}{t} \int_0^t  a(s) ds < \infty$ and is finite, the trivial solution to the deterministic equation $dX(t) = a(t)X(t)dt$ is (exponentially) unstable. By choose $|\alpha|$ sufficiently large such that $(a,\alpha)\in \mathbb{V}_{\mathsf{nonaut}}$, this trivial solution to \eqref{SDE:Continuous0} is exponentially stable almost surely. This is called \textit{almost sure stochastic stabilisation}, see e.g. \cite{arnold1983stabilization}. The idea of capturing the same stability/instability for the Milstein scheme of \eqref{SDE:Continuous0} will be investigated in our work.

\medskip

 Let $\{t_n^h\}_{n\ge 0}$ be the discrete-time sequence defined by $ t_{n}^h :=  nh$ for $n\in \mathbb{N}$, where $h$ is a fixed step size. To aid the readability, we write simply $t_n$ instead of $t_n^h$, once $h$ is already given. The so-called Milstein approximation of the solution $X(t)$ to  \eqref{SDE:Continuous0}, see e.g. \cite{kloeden1992stochastic,ngo2016strong} is the adapted stochastic process $X_h(t)$,  $t>0$, defined by linear interpolation of solutions to the stochastic difference equation
\begin{align}
         \left\{ \begin{array}{lll}
         X_h(t_{n+1}) = \displaystyle X_h(t_{n}) \left [ 1+ \Delta A_n-\dfrac{\alpha^2}{2} h  + \alpha\Delta W_n + \dfrac{\alpha^2}{2} \Delta W_n^2  \right ] ,\\
         X_h(0)=x_0,
         \end{array} \right. 
    \label{SDE:MScheme:Gen}
\end{align}
for each $t_{n-1} < t \le t_n$, $\Delta A_n=\int_{t_n}^{t_{n+1}}a(s)ds$ and $\Delta W_n = W(t_{n+1}) - W(t_{n})$. 

\medskip

The convergence of this approximation to the solution of \eqref{SDE:Continuous0} as $h\to 0$ is classical, see \cite[Theorem 10.6.3]{kloeden1992stochastic}. As said earlier, it is certainly desirable that the aforementioned stabilisation effect can be achieved in this discretised setting, since practically we use approximated solutions. This generally motivates of the question of structurally preserving numerical schemes, i.e. numerical methods that preserve certain important physical properties of the continuum setting, and this issue has been extensively investigated in the last decades, see \cite{buckwar2010towards,  mao1994stochastic, caraballo2003stochastic, ngo2020semi, kieu2024strong}.  . 
Results regarding numerical schemes for SDE such as \eqref{SDE:Continuous0} or related problems, the preservation of stochastic stabilisation for the Euler-Maruyama
scheme of SDEs has been developed in \cite{
higham2000mean,higham2007almost,wu2010almost,mora2017stable,szpruch2018V,li2021strong}. One of the main goal in these works is to estimate the Lyapunov exponents, which describes the large time behaviour of the dynamical systems. By using either Euler-Maruyama or Milstein schemes, these works obtain {\it upper bounds} of the Lyapunov exponents. To capture the precise dynamical property of the system, a lower or sharp bound of Lyapunov exopnents are desirable, yet this question has not been well addressed in the literature. To the best of our knowledge, this has been investigated only recently in the work of the first author \cite{hue2026sharp}, where upper and lower bounds of Lyapunov exponents for Milstein scheme of linear autonomous SDE were obtained. The current paper extends these results to non-autonomous and nonlinear problems. In order to do that, we exploit some stabilising terms featured by the Milstein scheme, combine them with precise bounds of Taylor expansions of polynomial and logarithmic functions, as well as establish precise estimates for the expectations of random processes induced by the numerical schemes. More details are given in the next subsection. Let us close this subsection by mentioning the recent work \cite{elgrou2026long}, where the preservation of stabilisation of the Euler-Maruyama scheme has been shown for some SPDE in Hilbert spaces.

\subsection{Main results and key ideas}
\label{Sec:MainResults}
We start with the estimate of Lyapunov exponents almost surely for the Milstein scheme of the linear non-autonomous equation \eqref{SDE:Continuous0}. For this, and subsequent results, we impose the following condition on the non-autonomous function: there exists $h_0 > 0$ such that, for all $0<h<h_0$ and $t_k = kh$, $k=0,1,2,\ldots$, it holds 
 \begin{equation}\label{lowerbound_a}
   \gamma_*:= \inf_{k\ge 0}\Delta A_k =  \inf_{k\ge 0}\int_{t_k}^{t_{k+1}}a(s)ds > -\frac 12 \quad \text{and} \quad \sup_{k\ge 0} \Delta A_k = \sup_{k\ge0}\int_{t_k}^{t_{k+1}}a(s)ds< +\infty.
 \end{equation}
 This is an easy-to-satisfy condition, e.g. when $a$ is uniformly continuous and bounded from below by a negative constant. For the sake of convenience, we write $\stackrel{\text{a.s.}}{=}$ or $\stackrel{\text{a.s.}}{\le}$ to indicate a computation or an estimate holds almost surely. 
\begin{theorem}
\label{Thm1}
Assume \eqref{lowerbound_a}. If \eqref{finiteness_a} holds, then there exists $0<h^{*}<1$ such that, for all $0< h < h^{*}$,
\begin{align*}
    \limsup_{n\to \infty} \frac{1}{t_n} \log |X_h(t_n)|  
    & \stackrel{\text{\normalfont a.s.}}{\le} \limsup_{n\to \infty} \left(  \frac{1}{n}\sum_{k=0}^{n-1} \dfrac{\log |1+ \Delta A_k |}{h}\right)   + \dfrac{\alpha^2}{2}\limsup_{n\to\infty}\left(\frac{1}{n}\sum_{k=0}^{n-1} \dfrac{1}{(1+ \Delta A_k )^2}\right)\\
   & \qquad  - \alpha^2\liminf_{n\to\infty}\left(\frac{1}{n}\sum_{k=0}^{n-1} \dfrac{1}{(1+ \Delta A_k )^2}\right) + O(h)
\end{align*} 
and 
\begin{align*}
    \liminf_{n\to \infty} \frac{1}{t_n} \log |X_h(t_n)| 
     &\stackrel{\text{\normalfont a.s.}}{\ge} \liminf_{n\to \infty} \left(  \frac{1}{n}\sum_{k=0}^{n-1} \dfrac{\log |1+ \Delta A_k |}{h}\right)   + \frac{\alpha^2}{2}\liminf_{n\to\infty}\left(\dfrac{1}{n}\sum_{k=0}^{n-1} \dfrac{1}{(1+ \Delta A_k )^2} \right)\notag \\
    & \quad\quad   - \alpha^2\limsup_{n\to\infty}\left(\dfrac{1}{n}\sum_{k=0}^{n-1} \dfrac{1}{(1+ \Delta A_k )^2} \right) + O(h^{\frac{1}{2}}).
\end{align*}  
If \eqref{finiteness_a} does not hold, then 
\begin{equation*}
	\liminf_{n\to \infty} \frac{1}{t_n} \log |X_h(t_n)| = +\infty.
\end{equation*}
\end{theorem}

\begin{remark}
    When $\lim_{n\to\infty}\frac{1}{n}\sum_{k=0}^{n-1}\frac{1}{(1+\Delta A_k)^2}$ exists, e.g. when $\lim_{k\to\infty}\Delta A_k = \Delta \in \mathbb R$, then the bounds in Theorem \ref{Thm1} can be shortened as
    \begin{equation*}
    \limsup_{n\to \infty} \frac{1}{t_n} \log |X_h(t_n)|  
     \stackrel{\text{a.s.}}{\le} \limsup_{n\to \infty} \left(  \frac{1}{n}\sum_{k=0}^{n-1} \dfrac{\log |1+ \Delta A_k |}{h}\right)  - \dfrac{\alpha^2}{2}\lim_{n\to\infty}\left(\frac{1}{n}\sum_{k=0}^{n-1} \dfrac{1}{(1+ \Delta A_k )^2}\right) + O(h)
    \end{equation*}
    and
    \begin{equation*}
            \liminf_{n\to \infty} \frac{1}{t_n} \log |X_h(t_n)| 
     \stackrel{\text{a.s.}}{\ge} \liminf_{n\to \infty}\left(  \frac{1}{n}\sum_{k=0}^{n-1} \dfrac{\log(1+ \Delta A_k)}{h}\right)  - \dfrac{\alpha^2}{2}\lim_{n\to\infty} \left( \frac{1}{n}\sum_{k=0}^{n-1} \dfrac{1}{(1+ \Delta A_k )^2} \right)  + O(h^{\frac{1}{2}}).
    \end{equation*}
\end{remark}
Let's fix $t$ and let $h\to 0$, we can use $nh = t$ and the approximation $\log(1+x) \approx x$ as $x\to 0$ to see that
\begin{equation*}
    \frac{1}{n}\sum_{i=0}^{n-1}\frac{\log(1+\Delta A_k)}{h} \approx \frac{1}{t}\sum_{k=0}^{n-1}\int_{t_k}^{t_{k+1}}a(s)ds \xrightarrow{h\to 0} \frac{1}{t}\int_0^ta(s)ds
\end{equation*}
and
\begin{equation*}
    \frac{\alpha^2}{2n}\sum_{k=0}^{n-1}\frac{1}{(1+\Delta A_k)^2} \xrightarrow{h\to 0} \frac{\alpha^2}{2}.
\end{equation*}
Therefore, when $h\to 0$ the estimates in Theorem \ref{Thm1} recover the precise Lyapunov exponent for the continuum problem in \eqref{Lya:Compute}, as long as \eqref{finiteness_a} is fulfilled. When the Lyapunov exponent in \eqref{Lya:Compute} is infinite, Theorem \ref{Thm1} shows that it is also the case for the discretised level for sufficiently small $h>0$. 

\medskip
Let $0<p<1$, we also show that the Lyapunov exponents for the Milstein scheme of \eqref{SDE:Continuous0} can be estimated in $p$-moment, which is stated in the following theorem.
 \begin{theorem}
 \label{Thm2}
     Assume \eqref{lowerbound_a}.   If \eqref{finiteness_a} holds, then for each $p\in (0,1)$, there exists $0<h^*<1$ such that, for all $0< h < h^{*}$,
\begin{align*}
    \limsup_{n\to \infty}\dfrac{1}{t_n}\log \mathbb{E} |X_h(t_n)|^p &\le \limsup_{n\to \infty}\dfrac{1}{n} \Bigg[ \sum_{k=0}^{n-1}  \dfrac{1}{h} \log \left( (1+ \Delta A_k)^p - p(1-p) \dfrac{1}{(1+ \Delta A_k)^{2-p}} \alpha^2h + O(h^2) \right)  \\
   & \qquad \quad  +  \sum_{k=0}^{n-1} \dfrac{p(1-p) (1+ \Delta A_k)^{p-2} \alpha^2}{2\left((1+ \Delta A_k)^p - p(1-p) (1+ \Delta A_k)^{p-2} \alpha^2h + O(h^2)\right)}  \Bigg ] +O(h)
\end{align*}
and 
\begin{align*}
  &  \liminf_{n\to \infty}\dfrac{1}{t_n}\log \mathbb{E} |X_h(t_n)|^p \notag \\
  & \ge  \liminf_{n\to \infty}\dfrac{1}{n} \Bigg[ \sum_{k=0}^{n-1}  \dfrac{1}{h} \log \left( (1+ \Delta A_k)^p - p(1-p) \dfrac{1}{(1+ \Delta A_k)^{2-p}} \alpha^2h + O(h^{\frac{3}{2}})  \dfrac{1}{(1+\Delta A_k)^{3-p}} \right)  \notag\\
   & \quad \qquad   +   \sum_{k=0}^{n-1} \dfrac{p(1-p) (1+ \Delta A_k) \alpha^2}{2\left((1+ \Delta A_k)^3 - p(1-p) (1+ \Delta A_k) \alpha^2h + O(h^{\frac{3}{2}})  \right)}  \Bigg] +O(h).
\end{align*}  
If \eqref{finiteness_a} does not hold, we have
\begin{equation*}
    \liminf_{n\to \infty}\dfrac{1}{t_n}\log \mathbb{E} |X_h(t_n)|^p = +\infty.
\end{equation*}
 \end{theorem}
Let us describe the ideas to prove these results. For Theorem \ref{Thm1}, by using the rescursive relation in \eqref{SDE:MScheme:Gen}, we can rewrite 
\begin{equation*}
    \frac{1}{t_n}\log|X_h(t_n)| = \frac{1}{t_n}\log|x_0| + \frac{1}{2t_n}\sum_{k=0}^{n-1}\log\Big((1+\Delta A_k)^2 + N_k\Big)
\end{equation*}
where the random process $N_k$ is defined as
\begin{equation*}
    N_k = 2(1+\Delta A_k)\bra{Y_k - \frac{\alpha^2}{2}h} + \bra{Y_k - \frac{\alpha^2}{2}h}^2 \quad \text{ with } \quad Y_k = \alpha \Delta W_k + \frac{\alpha^2}{2}(\Delta W_k)^2.
\end{equation*}
The upper bound of the Lyapunov exponent can be obtained following the ideas of e.g. \cite{higham2007almost}, where upper estimates of the logarithmic functions are utilised. The lower bound is trickier and needs more careful analysis due to the singularity of the logarithmic function near zero. Here we exploit the characteristic of the Milstein scheme to first get the lower bound $Y_k \ge -1/2$ a.s. This in turn helps to get crucial bound $N_k \ge -\frac{m-1}{m}(1+\Delta A_k)^2$ a.s. for a certain constant $m>1$ (see Lemma \ref{Le:Est:Nk}), which allows to bound the Taylor expansion of $\log((1+\Delta A_k)^2 + N_k)$ from below by a third order polynomial. Finally, in order to treat the nonlinear terms arising from this polynomial, we compute explicitly the expectations $\mathbb{E}(N_k^j)$ for $j=1,2,3$ as well as estimate $\mathbb{E}(\xi_\rho(N_k))$  for a suitable function $\xi_{\rho}(\cdot)$. The proof for the $p$-moment in Theorem \ref{Thm2} follows similar ideas, where the modifications concern the Taylor expansion of $(x+\eta)^{p/2}$ near $x = \eta$ as well as the estimates of the expectation $\mathbb{E}(\Psi_\rho(N_k))$ where $\Psi_{\rho}(\cdot)$ is related to the this aforementioned Taylor expansion.

\medskip
We believe that the above ideas are robust and can be applied to other problems. We demonstrate this by considering the nonlinear non-autonomous SDE
of finding $X: \mathbb{R}_+ \times \Omega \to \mathbb{R}$ such that
\begin{equation}
	dX(t) = a(t)X(t)dt + f(t, X(t))dW(t)  \quad \text{for } t>0, \qquad 
	X(0) = x_0, 
	\label{SDE:Continuous0_non}
\end{equation}
where $W$ is a standard Brownian motion on a filtered probability space $(\Omega,\mathcal F, \mathbb{F}, \mathbb{P})$, $x_0 \not = 0$ is the initial data, $f$ is a function of $(t,X)$ and $a(t)$ is a function of time. We assume the following assumption on $f$:

\medskip
\noindent {\bf Assumption (F)}: The function $f: \mathbb R\times \mathbb R \to \mathbb R$ satisfies: there exist positive constants $f_*, f^*$ such that
\begin{equation}\label{F}\tag{\textbf{F}}
	f_* \le \frac{f(t,x)}{x} \le f^*, \quad \max\bigg\{\dfrac{1}{2},\frac{2}{3+2\gamma_*} \bigg\}\leq  \frac{ x \partial_xf(t,x)}{f(t,x)} < +\infty, \quad \forall x\in \mathbb R\backslash\{0\}, \forall t\in \mathbb R.
\end{equation}
\medskip
From the assumption \eqref{F}, it follows that there exist constants $M_1, M_2$ such that 
$$M_1 \le \dfrac{f(t,x) \partial_x f(t,x)}{x} \le M_2, \quad \forall x\in \mathbb R\backslash\{0\}, \forall t\in \mathbb R.$$ 
Simple cases of $f$ satisfying \eqref{F} include functions which are bounded uniformly in $t$ and grows at most linearly in $x$, e.g. $f(t, x) = \frac{(2+t^2)(1+x^2)x}{(1+t^2)(2+x^2)}$. The stabilisation of \eqref{SDE:Continuous0_non} has been shown in even more general setting, see e.g. \cite{caraballo2001stabilization}. Note that it is not clear therein, that the Lyapunov exponents (in the continuous setting) are optimal since only estimates depending on parameters concerning $f$ and $a$ are given, see e.g. \cite[Theorem 2.1]{caraballo2001stabilization}. Here, by exploiting the above ideas, we provide some explicit upper and lower estimates of the Lyapunov exponents for the Milstein discretised scheme. This scheme for \eqref{SDE:Continuous0_non} reads as
\begin{align*}
	X_h(t_{n+1})
	 = X_h(t_n) \left[ 1+ \Delta A_n + \dfrac{f(t_n, X_h(t_n))}{X_h(t_n)} \Delta W_n +\dfrac{1}{2} \dfrac{f(t_n, X_h(t_n)) \partial_xf(t_n, X_h(t_n))}{X_h(t_n)} [(\Delta W_n)^2-h] \right].
\end{align*}
We denote $\gamma_{1n}:=\dfrac{f(t_n, X_h(t_n))}{X_h(t_n)}$ and $\gamma_{2n}=\dfrac{f(t_n, X_h(t_n)) \partial_xf(t_n, X_h(t_n))}{X_h(t_n)}$. Thus, we have  $f_* \leq \gamma_{1n} \leq f^*$ and $M_1 \le \gamma_{2n} \le M_2 $, and $X_h(t_n)$ can be rewritten
\begin{align}
	\label{Nonl}
	X_h(t_{n}) & = x_0 \prod_{k=0}^{n-1} \Big[  1 + \Delta A_k +\gamma_{1k} \Delta W_k + \dfrac{1}{2} \gamma_{2k} [(\Delta W_k)^2-h] \Big].
\end{align}
By utilising the above ideas with suitable modifications, we obtain the almost sure and $p$-moment estimates of the Lyapunov exponents for the Milstein scheme of the nonlinear problem in the following theorems.
\begin{theorem}
\label{Thm3}
Assume \eqref{finiteness_a}, \eqref{lowerbound_a} and \eqref{F}.
There exists $0<h^{*}<1$ such that 
\begin{equation}\label{h1}
\begin{aligned}
    \limsup_{n\to \infty} \frac{1}{t_n} \log |X_h(t_n)| 
    & \stackrel{\text{\normalfont a.s.}}{\le} \limsup_{n\to \infty} \left(  \frac{1}{n}\sum_{k=0}^{n-1} \dfrac{\log |1+ \Delta A_k |}{h}\right)  + \frac{f^{*2}}{2}\limsup_{n\to\infty}\frac{1}{n}\sum_{k=0}^{n-1}\frac{1}{(1+\Delta A_k)^2}\\
    &\qquad  - f_*^2\liminf_{n\to\infty}\frac{1}{n}\sum_{k=0}^{n-1}\frac{1}{(1+\Delta A_k)^2} + O(h)   
\end{aligned}
\end{equation}
and 
\begin{equation}\label{h2}
\begin{aligned}
   \liminf_{n\to \infty} \frac{1}{t_n} \log |X_h(t_n)| 
    & \ge \liminf_{n\to \infty} \left(  \frac{1}{n}\sum_{k=0}^{n-1} \dfrac{\log |1+ \Delta A_k |}{h}\right)  + \frac{f_*^{2}}{2}\liminf_{n\to\infty}\frac{1}{n}\sum_{k=0}^{n-1}\frac{1}{(1+\Delta A_k)^2}\\
    &\qquad  - f^{*2}\limsup_{n\to\infty}\frac{1}{n}\sum_{k=0}^{n-1}\frac{1}{(1+\Delta A_k)^2} + O(h^{\frac 12})    
\end{aligned}
\end{equation}
for all $0< h < h^{*}.$
\end{theorem} 
Here, we also see that for fixed $t = nh$, if we first let $h\to 0$ then $t\to\infty$, the right hand side of \eqref{h1} tends to
\begin{equation*}
    \lim_{t\to\infty}\frac{1}{t}\int_0^ta(s)ds + \frac{f^{*2}}{2} - f_*^2 
\end{equation*}
while the right hand side of \eqref{h2} converges to
\begin{equation*}
    \lim_{t\to\infty}\frac{1}{t}\int_0^ta(s)ds + \frac{f_{*}^{2}}{2} - {f^*}^2 
\end{equation*}
which are explicit upper and lower bounds for the Lyapunov exponent of the continuous SDE \eqref{SDE:Continuous0_non}. In other words, our estimates for the Milstein scheme provides explicit two-sided bounds for the Lyapunov exponent in the discretised setting. Note that due to the implicit nonlinearity, the optimal bound of Lyapunov exponent, even for the SDE \eqref{SDE:Continuous0_non} is unknown, and hence such a sharp bound for the Milstein scheme remains as an interesting open problem. Similarly to Theorem \ref{Thm2}, we also have the $p$-moment estimates of Lyapunov exponents for Milstein scheme of the nonlinear problem \eqref{SDE:Continuous0_non}.
\begin{theorem}
\label{Thm4}
Assume \eqref{finiteness_a}, \eqref{lowerbound_a} and \eqref{F}. Then for any $0<p<1$, there exists $0<h^{*}<1$ such that 
\begin{align*}
   &\limsup_{n\to \infty}\dfrac{1}{t_n}\log \mathbb{E} |X_h(t_n)|^p \notag \\
   & \le \limsup_{n\to \infty}\dfrac{1}{n} \Bigg[ \sum_{k=0}^{n-1}  \dfrac{1}{h} \log \left( (1+ \Delta A_k)^p + \dfrac{p}{2} \dfrac{1}{(1+ \Delta A_k)^{2-p}} (f^{*2} - (3-2p)f^2_{*}) h + O(h^2) \right) \notag \\
   & \qquad\qquad\qquad\quad    +  \sum_{k=0}^{n-1} \dfrac{p(1-p) f^2_{*} }{2(1+ \Delta A_k)^2 + p  (f^{*2} - (3-2p)f^2_{*}) h +2(1+ \Delta A_k)^{2-p}  O(h^2)}  \Bigg ] +O(h)
\end{align*}
and 
\begin{align*}
  &  \liminf_{n\to \infty} \dfrac{1}{t_n}\log \mathbb{E} |X_h(t_n)|^p \notag \\
  & \ge   \liminf_{n\to \infty} \dfrac{1}{n} \Bigg[ \sum_{k=0}^{n-1}  \dfrac{1}{h} \log \left( (1+ \Delta A_k)^p + \dfrac{p}{2} \dfrac{1}{(1+ \Delta A_k)^{2-p}} (f_*^2 - (3-2p)f^{*2}) h + O(h^{\frac{3}{2}}) \dfrac{1}{(1+\Delta A_k)^{3-p}} \right)  \notag\\
   & \quad \quad  \quad \quad  \quad \quad  +  \sum_{k=0}^{n-1} \dfrac{p(1-p) (1+\Delta A_k) f^{*2} }{2(1+ \Delta A_k)^3 + p (1+\Delta A_k) (f_*^2 - (3-2p)f^{*2}) h +  O(h^{\frac{3}{2}})}  \Bigg] +O(h)
\end{align*} 
for all $0< h < h^{*}.$    
\end{theorem}

\textbf{Organization of the paper:} In the next section, we present some preliminaries concerning the Brownian motion and the Kolmogorov law of large numbers which are needed for our analysis. Section \ref{sec:linear} focuses on the Milstein scheme for the linear problem, where we start with lower bounds of the random process $N_k$ as well as some expectation estimates. The proofs of Theorems \ref{Thm1} and \ref{Thm2} are presented in the subsequent subsections. Section \ref{sec:nonlinear} focuses on the Milstein scheme for the nonlinear problem \eqref{SDE:Continuous0_non}, where lower bounds of some random process $\mathcal{N}_k$ and expectation estimates are given in the first subsection, while the proofs of Theorem \ref{Thm3} and \ref{Thm4} are presented at the last two subsections. Finally, we present in Section \ref{sec:num} some numerical simulations to illustrate our theoretical results.

\section{Preliminaries} 

We recall basic computations related to the Brownian motion's increments, including their expectations in Lemma \ref{Lem:BasicProp_0}.

\begin{lemma} [\cite{de9805d084f3456ea59ea3429c92b726}, Theorem 2.5]
For all $0\le s\le t<\infty$ and $k\in \mathbb{N}$, $k\ge 1$,
\begin{align*}
    \mathbb{E} \big[  (W_t - W_s)^{2k} \big] =  (2k-1)!! (t-s)^k \quad \text{and} \quad  \mathbb{E}  \big[ (W_t - W_s)^{2k-1} \big] =0, 
\end{align*}
where the notation $(2k-1)!!$ denotes the product of all odd integer from $1$ to $(2k-1)$. 
\label{Lem:BasicProp_0}
\end{lemma}



\begin{lemma} 
\label{Lem:Kolmo}
Assume that $\{Y_i|\,i\ge 1\}$ is a sequence of independent random variables with the finite variances $\mathbb{V}(Y_i)$, and $\{b_n\}$ is a positive sequence such that $\lim_{n\to \infty} b_n = \infty$. If  $\sum_{i=1}^\infty \frac{1}{b^2_i} \mathbb{V}(Y_i) <\infty \quad $ then 
    \begin{align}
    \label{limsup}
      \limsup_{n\to \infty}  \frac{1}{b_n} \sum_{i=1}^n Y_i \stackrel{\text{\normalfont a.s.}}{\le } \limsup_{n\to \infty} \frac{1}{b_n}  \sum_{i=1}^n \mathbb{E}(Y_i); 
      \end{align} 
   and
    \begin{align}
    \label{liminf}
      \liminf_{n\to \infty}  \frac{1}{b_n} \sum_{i=1}^n Y_i \stackrel{\text{\normalfont a.s.}}{\ge } \liminf_{n\to \infty} \frac{1}{b_n}  \sum_{i=1}^n \mathbb{E}(Y_i).  
    \end{align}
\end{lemma} 
\begin{proof}
Applying Kolmogorov's First Strong Law of Large Numbers, we get
\begin{align*}
 \lim_{n\to \infty} \left(  \frac{1}{b_n} \sum_{i=1}^n Y_i - \frac{1}{b_n}  \sum_{i=1}^n \mathbb{E}(Y_i) \right) \stackrel{\text{\normalfont a.s.}}{=}0. 
\end{align*}
This implies that 
\begin{align*}
    \left\{ \begin{array}{l}
    \limsup\limits_{n\to \infty} \left(  \frac{1}{b_n} \sum_{i=1}^n Y_i - \frac{1}{b_n}  \sum_{i=1}^n \mathbb{E}(Y_i) \right) \stackrel{\text{\normalfont a.s.}}{=} 0;\\
    \liminf\limits_{n\to \infty} \left(  \frac{1}{b_n} \sum_{i=1}^n Y_i - \frac{1}{b_n}  \sum_{i=1}^n \mathbb{E}(Y_i) \right) \stackrel{\text{\normalfont a.s.}}{=} 0.
    \end {array} \right.
\end{align*}
By the properties of $\limsup$ and $\liminf$, we obtain
\begin{align*}
  \limsup_{n\to \infty}  \frac{1}{b_n} \sum_{i=1}^n Y_i &\stackrel{\text{\normalfont a.s.}}{\le } \limsup_{n\to \infty} \left(  \frac{1}{b_n} \sum_{i=1}^n Y_i - \frac{1}{b_n}  \sum_{i=1}^n \mathbb{E}(Y_i) \right)  + \limsup_{n\to \infty} \frac{1}{b_n}  \sum_{i=1}^n \mathbb{E}(Y_i)  \\
  &\stackrel{\text{\normalfont a.s.}}{= }  \limsup_{n\to \infty} \frac{1}{b_n}  \sum_{i=1}^n \mathbb{E}(Y_i);
\end{align*}
and
\begin{align*}
  \liminf_{n\to \infty}  \frac{1}{b_n} \sum_{i=1}^n Y_i &\stackrel{\text{\normalfont a.s.}}{\ge } \liminf_{n\to \infty} \left(  \frac{1}{b_n} \sum_{i=1}^n Y_i - \frac{1}{b_n}  \sum_{i=1}^n \mathbb{E}(Y_i) \right)  + \liminf_{n\to \infty} \frac{1}{b_n}  \sum_{i=1}^n \mathbb{E}(Y_i)  \\
  &\stackrel{\text{\normalfont a.s.}}{= }  \liminf_{n\to \infty} \frac{1}{b_n}  \sum_{i=1}^n \mathbb{E}(Y_i).
\end{align*}
This finishes the proof.
\end{proof}
 
We present a truncation technique for Taylor's expansion, where we will bound from both sides the following functions, for $r\in (0,1)$,
\begin{align*}
    f_{\eta,r}(x) = (\eta+x)^r \text{ for } x>-\eta,  \quad \text{and} \quad g_{\rho}(x) = \log(\rho+x) \text{ for } x>-\rho 
\end{align*}
by considering suitable truncations of their Taylor's expansions together with some modification. For given positive real numbers $\eta,\rho$, formal Taylor's expansions of these functions are 
\begin{align*}
    f_{\eta,r}(x) &=    \sum_{j=0}^\infty 
    \begin{pmatrix}
    r \\
    j 
    \end{pmatrix}
    \eta^{q-j}x^j = \eta^r +  \frac{1}{\eta^{1-r}} r x + 
    \frac{1}{\eta^{2-r}} \begin{pmatrix}
    r \\
    2
    \end{pmatrix}  x^2 + \frac{1}{\eta^{3-r}} \begin{pmatrix}
    r \\
    3 
    \end{pmatrix} 
     x^3 + \frac{1}{\eta^{4-r}} \begin{pmatrix}
    r \\
    4 
    \end{pmatrix} 
     x^4 + \dots, \\
    g_{\rho}(x) & = \log \rho +     \sum_{l=1}^\infty 
    \frac{(-1)^{l-1}}{l\rho^{\hspace{0.03cm}l}} x^l = \log \rho + \frac{x} {\rho}  - \frac{x^2}{2\rho^2}  +   \frac{x^3}{3\rho^3}  - \frac{x^4}{4\rho^4}  + \dots,
\end{align*}
where $\begin{pmatrix}
    r \\
    j 
    \end{pmatrix}=\dfrac{r(r-1)\dots(r-j+1)}{j!}$, which is $1$ if $j=0$ and $r$ if $j=1$. Regarding the upper bounds of these functions, it is well-known that
    \begin{align} 
    \left\{ \begin{array}{llll}
    f_{\eta,r}(x)  \le f_{\eta,r}^{\mathsf{up}}(x) :=  \eta^r +  \dfrac{1}{\eta^{1-r}} r x + 
    \dfrac{1}{\eta^{2-r}} \begin{pmatrix}
    r \\
    2
    \end{pmatrix}  x^2 + \dfrac{1}{\eta^{3-r}} \begin{pmatrix}
    r \\
    3 
    \end{pmatrix} 
     x^3 & \text{if } x>-\eta,
    \vspace{0.15cm}\\
    g_{\rho}(x)  \le g_{\rho}^{\mathsf{up}}(x):=  \log \rho + \dfrac{x}{\rho} - \dfrac{x^2}{2\rho^2} +   \dfrac{x^3}{3\rho^3}      & \text{if } x>-\rho, 
    \end{array} \right.
    \label{Inq:Up}
\end{align}
where we have truncated the above Taylor expansions up to the third order of $x$. Similarly, one can check for only $x\ge 0$ that  $f_{\eta,r}$ and $g_{\rho}$ are bounded from below by their truncated Taylor expansions up to the fourth order. However, for $x<0$, lower bounds of $f_{\eta,r}$, $0<r<1$, and $g_{\rho}$ cannot be obtained in the same way since all terms in their Taylor's expansions (except the first terms) are negative. 

\medskip

To find lower bounds, we avoid the singularity of $f_{\eta,r}$ and $g_{\rho}$ at the left limits by respectively restricting $x$ on $-\frac{m-1}{m} \eta <x<0$ and $-\frac{m-1}{m} \rho<x<0$ for suitable $m>1$. This will be proved in the following lemma.

\begin{lemma} 
\label{Lem:TayTech} 
Let $\eta>0$, $0<r<1$, $\rho>0$ and $m>1$. Then,
\begin{align*} 
\begin{array}{llll}
\displaystyle f_{\eta,r}^{\mathsf{low}}(x) \le f_{\eta,r}(x)  \le f_{\eta,r}^{\mathsf{up}}(x), & x > -\dfrac{m-1}{m}\eta, \vspace{0.15cm} \\
\displaystyle g_{\rho}^{\mathsf{low}}(x) \le g_{\rho}(x)  \le  g_{\rho}^{\mathsf{up}}(x), & x > -\dfrac{m-1}{m}\rho, 
\end{array}
\end{align*}
where $f_{\eta,r}^{\mathsf{up}}, g_{\rho}^{\mathsf{up}}$ are in \eqref{Inq:Up}, and the lower bounds are defined as
\begin{align*}
    f_{\eta,r}^{\mathsf{low}}(x):= \left\{ \begin{array}{llrll}
     \eta^r +  \dfrac{1}{\eta^{1-r}} r x + 
    \dfrac{1}{\eta^{2-r}} \begin{pmatrix}
    r \\
    2
    \end{pmatrix}  x^2 + \dfrac{1}{\eta^{3-r}} \begin{pmatrix}
    r \\
    3 
    \end{pmatrix} 
     x^3 + \dfrac{1}{\eta^{4-r}} \begin{pmatrix}
    r \\
    4 
    \end{pmatrix} 
     x^4 \qquad  \text{if } x \ge 0,
    \vspace{0.15cm}\\
    \eta^r +  \dfrac{1}{\eta^{1-r}} r x + 
    \dfrac{1}{\eta^{2-r}} \begin{pmatrix}
    r \\
    2
    \end{pmatrix}  x^2 + \left( \dfrac{m} {\eta}\right)^{3-r} \begin{pmatrix}
    r \\
    3 
    \end{pmatrix} 
     x^3 \quad  \text{if } - \dfrac{m-1}{m}\eta < x < 0, 
    \end{array} \right.
\end{align*}
and 
\begin{align*}
    g_{\rho}^{\mathsf{low}}(x):= \left\{ \begin{array}{llrll}
     \log \rho + \dfrac{x}{\rho} - \dfrac{x^2}{2\rho^2} +   \dfrac{x^3}{3\rho^3} -   \dfrac{x^4}{4\rho^4} & \text{if } x \ge 0,
    \vspace{0.15cm}\\
    \log \rho + \dfrac{x}{\rho} - \dfrac{x^2}{2\rho^2} +   m^3 \dfrac{x^3}{3\rho^3}   & \text{if } - \dfrac{m-1}{m}\rho < x < 0.  
    \end{array} \right.
\end{align*}
\end{lemma}
The proof of Lemma \ref{Lem:TayTech} is elementary and is provided in Appendix \ref{Apd.Lem:TayTech}.
\begin{remark}\label{remark:Taylor}
It follows directly from Lemma \ref{Lem:TayTech} the following weaker lower bounds
\begin{equation*}
    \widetilde f_{\eta,r}^{\mathsf{low}}(x)\le f_{\eta, r}(x)  \quad \text{ and } \quad \widetilde g_{\rho}^{\mathsf{low}}(x) \le g_{\rho}(x)
\end{equation*}
where
\begin{equation*}
    \widetilde f_{\eta,r}^{\mathsf{low}}(x) = \begin{cases}
        \eta^r + \dfrac{1}{\eta^{1-r}}rx + \dfrac{1}{\eta^{2-r}}\begin{pmatrix}r\\2 \end{pmatrix}x^2 + \dfrac{1}{\eta^{4-r}}\begin{pmatrix}r\\4 \end{pmatrix}x^4 & \text{if } x\ge 0,\\
        f_{\eta,r}^{\mathsf{low}}(x) & \text{if } -\dfrac{m-1}{m}\eta < x < 0
    \end{cases}
\end{equation*}
and
\begin{equation*}
    \widetilde g_{\eta,r}^{\mathsf{low}}(x) = \begin{cases}
        \log \rho + \dfrac{x}{\rho} - \dfrac{x^2}{2\rho^2} -   \dfrac{x^4}{4\rho^4} & \text{if } x\ge 0,\\
        g_{\eta,r}^{\mathsf{low}}(x) & \text{if } -\dfrac{m-1}{m}\eta < x < 0.
    \end{cases}
\end{equation*}
In our subsequent analysis, we will use these weaker lower bounds since their expectation estimates have the same order with the lower bounds given in Lemma \ref{Lem:TayTech}.
\end{remark}



\section{Linear non-autonomous equations}\label{sec:linear}

\subsection{Lower bounds of $N_k$ and expectation estimates}
For $k\in \mathbb{N}$, $k \ge 1$, and a given step size $h>0$, we will calculate the expectations of 
\begin{gather}
	{Y}_k := \alpha\Delta W_k +\dfrac{\alpha^2}{2}(\Delta W_k)^2, \quad  \text{and} \quad
	N_k := 2 \left( 1 + \Delta A_k \right) \left( {Y}_k -\dfrac{\alpha^2}{2} h  \right) +  \left( {Y}_k -\dfrac{\alpha^2}{2} h  \right)^2, 
	\label{Nk}
\end{gather}
in the following lemma. 
\begin{lemma} 
	\label{Lem:ComputeYkj} Let  $\jeven$ be the largest even number that is less than or equal to $j$. Then,  
	\begin{align*}
		\mathbb{E} ({Y}_k^j) =   \sum_{i=0}^{\frac{1}{2}\jeven} \begin{pmatrix}
			j \\
			2i 
		\end{pmatrix}  \dfrac{\alpha^{2j-2i}}{2^{j-2i}} [(2j-2i-1)!!] h^{j-i}, 
	\end{align*}
	and 
	\begin{align*}
		\mathbb{E} (N_k^l)  
		=   \left\{
		\begin{array}{llll}
			\alpha^2 h + \dfrac{1}{2}\alpha^4h^2   & \text{if} \quad l=1, \\
			4(1+ \Delta A_k )^2 \left( \alpha^2 h + \dfrac{1}{2}\alpha^4h^2  \right) + 4 (1+ \Delta A_k) \beta_3 + O(h^2)  & \text{if} \quad l=2, \\
			8 (1+ \Delta A_k )^3 \beta_3 + 12 (1+ \Delta A_k )^2 \beta_4 
			+ 6 (1+ \Delta A_k ) \beta_5 + O(h^2)   & \text{if} \quad l= 3,
		\end{array}
		\right.
	\end{align*}
	where \begin{align}
		\beta_m:= \mathbb{E} \left({Y}_k -\dfrac{\alpha^2}{2} h \right)^m = O(h^2) \quad \text{for} \quad  m=3,4,5. \label{betam}
	\end{align}
\end{lemma} 
We provide the complete proof of Lemma $\ref{Lem:ComputeYkj}$ in Appendix $\ref{Apd.Lem:ComputeYkj}$.

The following lemma is crucial in obtaining the lower bounds of Lyapunov exponents.
\begin{lemma} 
	\label{Le:Est:Nk}
	Assume \eqref{lowerbound_a}. There exists $m>1$ sufficiently large such that 
	\begin{align*}
		N_k  \stackrel{\text{a.s.}}{>} -\dfrac{m-1}{m} (1+\Delta A_k)^2 , \quad \forall  
		k \ge 0, 
	\end{align*}
	for any 
	\begin{align*}
		0<h < h_{\ast}:= \dfrac{2(1-1/ \sqrt{m}) (1+\gamma_{\ast})-1}{\alpha^2}. 
	\end{align*}  
\end{lemma}
\begin{proof} Under the assumption \eqref{lowerbound_a}, we have $1+\gamma_\ast >1/2$. Therefore, there exists a real number $m>1$ such that $1+\gamma_\ast > (1/2) (1-1/ \sqrt{m})^{-1}$, which consequently guarantees that $h_\ast>0$. Elementary computations give
\[
    Y_k = \frac 12(\alpha \Delta W_k + 1)^2 - \frac 12 \ge -\frac 12
\]
almost surely. Hence,  
    \begin{align*}
		& \dfrac{2\sqrt{m}-2}{m} (1+\Delta A_k)^2 + \frac{2}{\sqrt{m}} (1+\Delta A_k) \left( { Y}_k -\dfrac{\alpha^2}{2} h  \right) \\
		& \stackrel{\text{a.s.}}{\ge} \dfrac{2\sqrt{m}-2}{m} (1+\Delta A_k)^2 + \frac{2}{\sqrt{m}} (1+\Delta A_k) \left(-\frac{1}{2}-\dfrac{\alpha^2}{2} h \right)\\
        &\ge \frac{1}{\sqrt m}(1+\Delta A_k)\bra{2\bra{1 - \frac{1}{\sqrt m}}(1+\Delta A_k) - (1+\alpha^2h) }\\
        &\ge \frac{1}{\sqrt m}(1+\gamma_{\ast})\bra{2\bra{1 - \frac{1}{\sqrt m}}(1+\gamma_{\ast}) - 1 - \alpha^2h } > 0
	\end{align*}
	for all $0<h < h_\ast$. Using this estimate, for any integer $k\ge 0$ we now observe that  
	\begin{align*}
		N_k + \dfrac{m-1}{m} (1+\Delta A_k)^2 &=   \left[ 2 \left( 1 + \Delta A_k \right) \left( {Y}_k -\dfrac{\alpha^2}{2} h  \right) +  \left( {Y}_k -\dfrac{\alpha^2}{2} h  \right)^2 \right] + \dfrac{m-1}{m} (1+\Delta A_k)^2 \\ 
		& = \left(   \frac{\sqrt{m}-1}{\sqrt{m}}  \left( 1 + \Delta A_k \right) +  {Y}_k -\dfrac{\alpha^2}{2} h    \right)^2 + \dfrac{2\sqrt{m}-2}{m} (1+\Delta A_k)^2 \\
		& \quad + \frac{2}{\sqrt{m}}  \left( 1 + \Delta A_k \right) \left( {Y}_k -\dfrac{\alpha^2}{2} h  \right) \\
		& \stackrel{\text{a.s.}}{>} \dfrac{2\sqrt{m}-2}{m} (1+\Delta A_k)^2 + \frac{2}{\sqrt{m}}  \left( 1 + \Delta A_k \right) \left( {Y}_k -\dfrac{\alpha^2}{2} h  \right) \\
		& \stackrel{\text{a.s.}}{ > } 0,
	\end{align*}
	which shows the desired estimate for $N_k$. 
\end{proof}


\subsection{Almost sure stability - Proof of Theorem \ref{Thm1}}
To obtain sharp estimates of the Lyapunov exponent in the almost-sure sense, we need to bound the logarithmic function not only from above but also from below, for which we will employ the truncation technique given in Lemma \ref{Lem:TayTech}. The former will be constructed by adapting the method given in \cite{higham2007almost} and relying on the assumption about the difference $\Delta A_k$ for any $k$. However, the latter one, i.e., the estimate from below, requires us to precisely estimate the expectation of   
\begin{align}
	\xi_\rho(N_k(\omega)):=  
	\left\{ \begin{array}{llrll}
		 - \dfrac{N_k^4(\omega)}{4(1+ \Delta A_k )^8} & \text{if}& \omega \in \Omega_{k,1} &:=\{  N_k(\omega) \ge 0 \} ,
		\vspace{0.15cm}\\
		m^3\dfrac{N_k^3(\omega)}{3(1+ \Delta A_k )^6}  & \text{if}& \omega \in \Omega_{k,2} &:=\bigg\{ - \dfrac{m-1}{m} (1+ \Delta A_k )^2 < N_k(\omega) < 0 \bigg\} .  
	\end{array} \right.   
	\label{xirho}
\end{align}
from below.  Thanks to Lemma \ref{Le:Est:Nk}, with a sufficiently smooth scheme that $0<h < h_\ast$, we imply that for any integer $k\ge 0$ the sample space consists of $\Omega_{k,1}$ and $\Omega_{k,2}$, i.e.,
\begin{align*}
	\Omega = \Omega_{k,1} \cup \Omega_{k,2}, \quad \forall k\ge 0. 
\end{align*}
This allows us to use the probability density function of the normal distribution to estimate the expectation of $\xi_\rho(N_k)$ from below, where the integral with respect to the probability density function of $\xi_\rho(N_k)$ can be split into the sum of integrals on $\Omega_{k,1} $ and $ \Omega_{k,2}$. 

\begin{lemma}
	\label{Lem.E(xiNk)}
	For every $k\in \mathbb{N}$, it holds that
	\begin{align*}
		\mathbb{E} \left[\xi_\rho \left(N_k \right) \right]
	=  O(h^{\frac{3}{2}})     \frac{1}{(1+ \Delta A_k )^3}.
	\end{align*} 
\end{lemma}
\begin{proof} We first recall that ${Y}_k := \alpha\Delta W_k + (\alpha^2/2)(\Delta W_k)^2$. Using the definition of $N_k$, this stochastic process can be expanded in terms of $\Delta W_k$ as follows
	\begin{align}
		N_k &= 2 \left( 1 + \Delta A_k \right) \left( {Y}_k -\dfrac{\alpha^2}{2} h  \right) +  \left( {Y}_k -\dfrac{\alpha^2}{2} h  \right)^2\notag\\
		& = 2 \left( 1 + \Delta A_k \right) \left( \alpha\Delta W_k +\dfrac{\alpha^2}{2}(\Delta W_k)^2 -\dfrac{\alpha^2}{2} h  \right) +  \left( \alpha\Delta W_k +\dfrac{\alpha^2}{2}(\Delta W_k)^2 -\dfrac{\alpha^2}{2} h  \right)^2 \notag\\
		& = 2 \left( 1 + \Delta A_k \right) \left( \alpha\Delta W_k +\dfrac{\alpha^2}{2}(\Delta W_k)^2 \right) - \alpha^2h \left(\alpha\Delta W_k +\dfrac{\alpha^2}{2}(\Delta W_k)^2 \right) \notag  \\
		& \quad +  \left(\alpha\Delta W_k +\dfrac{\alpha^2}{2}(\Delta W_k)^2 \right)^2 +\left( \frac{\alpha^4}{4}h - \alpha^2  (1+  \Delta A_k ) \right) h . \notag 
	\end{align}
	By the property of the Brownian motion, $\Delta W_k \sim \mathcal N(0,h)$, or equivalently, $\Delta W_k/\sqrt{h}$ is a standard normal distribution. Hence, we can define the function
	\begin{align*}
		\mathcal H_k(y) &=  2 \left( 1 + \Delta A_k \right) \left( \alpha \sqrt{h} y +\dfrac{\alpha^2h}{2} y^2\right) - \alpha^2h \left( \alpha \sqrt{h} y +\dfrac{\alpha^2h}{2} y^2 \right)^2 \\
		& \quad  +  \left(\alpha \sqrt{h} y +\dfrac{\alpha^2h}{2} y^2\right)^2 +\left( \frac{\alpha^4}{4}h - \alpha^2  (1+  \Delta A_k ) \right) h, \quad y \in \mathbb{R},     
	\end{align*}
	and write shortly that  
	\begin{align*}
		N_k(\omega) = \mathcal H_k \left(\dfrac{\Delta W_k(\omega)}{\sqrt{h}} \right), \quad \forall  \omega \in \Omega,  k \ge 0.
	\end{align*}
	Thanks to Lemma \ref{Le:Est:Nk}, we see that $\mathcal H_k (\Delta W_k(\omega)/\sqrt{h})$ takes values at least $-((m-1)/m)(1+\Delta A_k)^2$ almost surely. Hence, the process $\xi_\rho(N_k)$ can be rewritten as 
	\begin{align*}
		\xi_\rho(N_k) = (\xi_\rho \circ \mathcal H_k ) \left(\dfrac{\Delta W_k}{\sqrt{h}} \right), \quad \forall k \ge 0.
	\end{align*}
	
	To estimate the expectation of the process $N_k$, in the following we will use the probability density function of $\Delta W_k/\sqrt{h}$. Indeed, we have 
	\begin{align*}
		& \mathbb{E} \left[\xi_\rho \left(N_k \right) \right] \\
		& = \frac{1}{2\pi} \int_{-\infty}^{+\infty} (\xi_\rho \circ \mathcal H_k )(y) e^{-\frac{y^2}{2}}dy \\
		& = \frac{1}{2\pi} \int_{\{ \mathcal H_k(y) \ge 0 \}} (\xi_\rho \circ \mathcal H_k )(y) e^{-\frac{y^2}{2}}dy  + \frac{1}{2\pi} \int_{\{ - \frac{m-1}{m} (1+ \Delta A_k )^2 < \mathcal H_k(y) < 0 \}} (\xi_\rho \circ \mathcal H_k )(y) e^{-\frac{y^2}{2}}dy\\
		&= \dfrac{m^3}{3(1+ \Delta A_k )^6 \sqrt{2\pi}} \int_{\{ - \frac{m-1}{m} (1+ \Delta A_k )^2 < \mathcal H_k(y) < 0 \}} \Big( \mathcal H_k(y)  \Big)^3  e^{-\frac{y^2}{2}}dy  \\
		& \quad - \dfrac{1}{4(1+ \Delta A_k )^8 \sqrt{2\pi}} \int_{\{ \mathcal H_k(y) \ge 0 \}} \Big( \mathcal H_k(y)  \Big)^4   e^{-\frac{y^2}{2}}dy  \notag\\ 
		& = \dfrac{m^3}{3(1+ \Delta A_k )^6 \sqrt{2\pi}} \int_{\{ - \frac{m-1}{m} (1+ \Delta A_k )^2 < \mathcal H_k(y) < 0 \}} \Bigg[ 2 \left( 1 + \Delta A_k \right) \left( \alpha \sqrt{h} y +\dfrac{\alpha^2h}{2} y^2\right)\notag\\
		& \quad - \alpha^2h \left( \alpha \sqrt{h} y +\dfrac{\alpha^2h}{2} y^2 \right)^2   +  \left(\alpha \sqrt{h} y +\dfrac{\alpha^2h}{2} y^2\right)^2 +\left( \frac{\alpha^4}{4}h - \alpha^2  (1+  \Delta A_k ) \right) h \Bigg]^3  e^{-\frac{y^2}{2}}dy \notag\\
		&\quad - \dfrac{1}{4(1+ \Delta A_k )^8 \sqrt{2\pi}} \int_{\{ \mathcal H_k(y) \ge 0 \}} \Bigg[ 2 \left( 1 + \Delta A_k \right) \left( \alpha \sqrt{h} y +\dfrac{\alpha^2h}{2} y^2\right)\notag\\
		&\quad  - \alpha^2h \left( \alpha \sqrt{h} y +\dfrac{\alpha^2h}{2} y^2 \right)^2  +  \left(\alpha \sqrt{h} y +\dfrac{\alpha^2h}{2} y^2\right)^2 +\left( \frac{\alpha^4}{4}h - \alpha^2  (1+  \Delta A_k ) \right) h \Bigg]^4  e^{-\frac{y^2}{2}}dy   \\
		& = O(h^{\frac{3}{2}})     \frac{1}{(1+ \Delta A_k )^3},
	\end{align*} 
    which is the desired estimate.
\end{proof}
We are now in a position to prove Theorem $\ref{Thm1}$.
\begin{proof}[Proof of Theorem $\ref{Thm1}$]
Inductively, it follows from the representation  \eqref{SDE:MScheme:Gen} that  
\begin{align*}
     X_h(t_{n+1}) = X_h(t_{n}) \left( 1 + \Delta A_n -\dfrac{\alpha^2}{2} h  +  Y_n  \right) = x_0 \prod_{k=0}^n \left(  1 + \Delta A_k -\dfrac{\alpha^2}{2} h +  {Y}_k   \right) . 
\end{align*}
This implies that
\begin{align}
     X_h(t_{n})  = x_0 \prod_{k=0}^{n-1} \left(  1 + \Delta A_k -\dfrac{\alpha^2}{2} h +  {Y}_k   \right) . 
     \label{Gen}
\end{align}
Therefore, by taking the logarithms of both sides of $\eqref{Gen}$, the discrete Lyapunov exponent is given by 
\begin{align}
     \frac{1}{t_n}\log |X_h(t_n)| =& \frac{1}{t_n}\log |x_0| +  \frac{1}{t_n}\sum_{k=0}^{n-1} \log  
 \Bigg|\left( 1 + \Delta A_k \right)  +  \left( {Y}_k -\dfrac{\alpha^2}{2} h  \right)  \Bigg| \notag \\
= &  \frac{1}{t_n}\log |x_0| +  \frac{1}{2t_n}\sum_{k=0}^{n-1} \log \left( (1+ \Delta A_k )^2 + N_k \right)  .
 \label{Lya:Compute}
 \end{align}
In the sequel, we will estimate this from both above and below, which will be based on the inequality \eqref{Inq:Up} and those of lower bound type given in Lemma \ref{Lem:TayTech}.

\medskip

\noindent \textit{\underline{Upper estimate}}: Using inequality \eqref{Inq:Up}, we get
\begin{align}
\frac{1}{2t_n}\log \left( (1+ \Delta A_k )^2 + N_k \right) \le \frac{1}{2t_n}\log (1+ \Delta A_k )^2 + \frac{1}{2t_n}\frac{N_k}{(1+ \Delta A_k )^2}\notag\\
\qquad\qquad -  \frac{1}{4t_n}\frac{N_k^2}{(1+ \Delta A_k )^4}   + \frac{1}{6t_n}\frac{N_k^3}{(1+ \Delta A_k )^6} . 
\label{a0}
\end{align}
Using $\eqref{limsup}$ and Lemma \ref{Lem:ComputeYkj}, we have 
 \begin{align}
\label{a1}
      \limsup_{n\to \infty} \frac{1}{2t_n}  \sum_{k=0}^{n-1} \frac{N_k}{(1+ \Delta A_k )^2}
    &  \stackrel{\text{a.s.}}{\le} \limsup_{n\to \infty} \left( \frac{\alpha^2}{2n}  \sum_{k=0}^{n-1} \frac{1}{(1+ \Delta A_k )^2}   + \frac{\alpha^4h}{4n}  \sum_{k=0}^{n-1} \frac{1}{(1+ \Delta A_k )^2} \right)\notag\\
   &    \le \limsup_{n\to \infty} \left( \frac{\alpha^2}{2n}  \sum_{k=0}^{n-1} \frac{1}{(1+ \Delta A_k )^2} \right)+ O(h)
\end{align}
 thanks to
\begin{equation*}
    \limsup_{n\to\infty}\frac{1}{n}\sum_{k=0}^{n-1}\frac{1}{(1+\Delta A_k)^2} < +\infty.
\end{equation*}
For the higher order terms concerning $N_k$ we recall the computation $\beta_m = \mathbb{E} \left({Y}_k -\dfrac{\alpha^2}{2} h \right)^m= O(h^2)$ (for $m\ge 3$)  from $\eqref{betam}$. This, together with $\eqref{liminf}$ and Lemma \ref{Lem:ComputeYkj}, yields 
\begin{align}
\label{a2}
    & \liminf_{n\to \infty}  \frac{1}{4t_n}  \sum_{k=0}^{n-1} \frac{N_k^2}{(1+ \Delta A_k )^4} \notag \\
       & \stackrel{\text{a.s.}}{\ge }  \liminf_{n\to \infty} \frac{1}{t_n}  \sum_{k=0}^{n-1} \frac{(1+ \Delta A_k )^2 ( \alpha^2 h +  \frac{1}{2}\alpha^4h^2  ) + (1+ \Delta A_k) \beta_3 + O(h^2) }{(1+ \Delta A_k )^4} \notag\\
       & {=}  \liminf_{n\to \infty}   \sum_{k=0}^{n-1} \left( \dfrac{\alpha^2}{n} \frac{1}{(1+ \Delta A_k )^2} + \dfrac{1}{2n} \dfrac{ \alpha^4 h}{(1+ \Delta A_k )^2} + \dfrac{1}{n} \dfrac{\beta_3}{h(1+ \Delta A_k )^3} + \dfrac{1}{n} \dfrac{O(h)}{(1+ \Delta A_k )^4} \right) \notag\\
      & {=}   \liminf_{n\to \infty} \left( \frac{\alpha^2}{n}  \sum_{k=0}^{n-1} \frac{1}{(1+ \Delta A_k )^2}\right) + O(h), 
\end{align}
 thanks to $\beta_3 = O(h^2)$ and 
\[
\limsup_{n\to\infty}\sum_{j=2}^4\frac{1}{n}\sum_{k=0}^{n-1}\frac{1}{(1+\Delta A_k)^j} < +\infty.
\]
Similarly, we can estimate
\begin{align}
\label{a3}
& \limsup_{n\to \infty} \frac{1}{6t_n}  \sum_{k=0}^{n-1} \frac{N_k^3}{(1+ \Delta A_k )^6} \notag\\
    & \stackrel{\text{a.s.}}{\le} \limsup_{n\to \infty}  \frac{1}{6t_n}  \sum_{k=0}^{n-1} \frac{  8 (1+ \Delta A_k )^3 \beta_3 + 12 (1+ \Delta A_k )^2 \beta_4 
     + 6 (1+ \Delta A_k ) \beta_5 + O(h^2) }{(1+ \Delta A_k )^6} \notag \\
     & \le Ch\limsup_{n\to\infty}\sum_{j=3}^6\frac{1}{n}\sum_{k=0}^{n-1}\frac{1}{(1+\Delta A_k)^j} \notag\\
   &=  O(h),
\end{align}
 where we used $\beta_m = O(h^2)$ for $m\in \{3,4,5\}$.
Therefore, combining $\eqref{a0}$, $\eqref{a1},$ $\eqref{a2}$ and $\eqref{a3}$, we obtain 
\begin{align*}
    \limsup_{n\to \infty} \frac{1}{t_n} \log |X_h(t_n)|  
    & \stackrel{\text{a.s.}}{\le} \limsup_{n\to \infty} \left(  \frac{1}{n}\sum_{k=0}^{n-1} \dfrac{\log |1+ \Delta A_k |}{h}\right)   + \dfrac{\alpha^2}{2}\limsup_{n\to\infty}\left(\frac{1}{n}\sum_{k=0}^{n-1} \dfrac{1}{(1+ \Delta A_k )^2}\right)\\
    &\quad  - \alpha^2\liminf_{n\to\infty}\left(\frac{1}{n}\sum_{k=0}^{n-1} \dfrac{1}{(1+ \Delta A_k )^2}\right) + O(h).
\end{align*}

\noindent \textit{\underline{Lower estimate}}: 
By Lemma $\ref{Le:Est:Nk}$ and Remark \ref{remark:Taylor}, we have
\begin{align}
\label{a0'}
\log \left( (1+ \Delta A_k )^2 + N_k \right) \ge \log (1+ \Delta A_k )^2 + \frac{N_k}{(1+ \Delta A_k )^2}  -  \frac{N_k^2}{2(1+ \Delta A_k )^4} + \xi_\rho(N_k) 
\end{align}
with $\xi_\rho(\cdot)$ is defined in \eqref{xirho}.
Using $\eqref{liminf}$ and Lemma \ref{Lem:ComputeYkj}, we have 
 \begin{align}
\label{a11}
      \liminf_{n\to \infty} \frac{1}{2t_n}  \sum_{k=0}^{n-1} \frac{N_k}{(1+ \Delta A_k )^2}
 & \stackrel{\text{a.s.}}{\ge} \liminf_{n\to \infty} \left( \frac{\alpha^2}{2n}  \sum_{k=0}^{n-1} \frac{1}{(1+ \Delta A_k )^2}   + \frac{\alpha^4h}{4n}  \sum_{k=0}^{n-1} \frac{1}{(1+ \Delta A_k )^2} \right)\notag \\
&  = \liminf_{n\to \infty} \left( \frac{\alpha^2}{2n}  \sum_{k=0}^{n-1} \frac{1}{(1+ \Delta A_k )^2}\right) + O(h).
\end{align}
For the higher order terms concerning $N_k$ we recall again the computation $\beta_m = \mathbb{E} \left({Y}_k -\dfrac{\alpha^2}{2} h \right)^m= O(h^2)$ (for $m\ge 3$)  from $\eqref{betam}$. This, together with $\eqref{limsup}$ and Lemma \ref{Lem:ComputeYkj}, yields 
\begin{align}
\label{a22}
    &  \limsup_{n\to \infty} \frac{1}{4t_n}  \sum_{k=0}^{n-1} \frac{N_k^2}{(1+ \Delta A_k )^4} \notag \\
       & \stackrel{\text{a.s.}}{\le }  \limsup_{n\to \infty}  \frac{1}{t_n}  \sum_{k=0}^{n-1} \frac{(1+ \Delta A_k )^2 ( \alpha^2 h +  \frac{1}{2}\alpha^4h^2  ) + (1+ \Delta A_k) \beta_3 + O(h^2) }{(1+ \Delta A_k )^4} \notag\\
       & \stackrel{\text{a.s.}}{=}  \limsup_{n\to \infty}    \sum_{k=0}^{n-1} \left( \dfrac{\alpha^2}{n} \frac{1}{(1+ \Delta A_k )^2} + \dfrac{1}{2n} \dfrac{ \alpha^4 h}{(1+ \Delta A_k )^2} + \dfrac{1}{n} \dfrac{\beta_3}{h(1+ \Delta A_k )^3} + \dfrac{1}{n} \dfrac{O(h)}{(1+ \Delta A_k )^4} \right) \notag\\
      & \stackrel{\text{a.s.}}{=}  \limsup_{n\to \infty} \left( \frac{\alpha^2}{n}  \sum_{k=0}^{n-1} \frac{1}{(1+ \Delta A_k )^2} \right) +  O(h). 
\end{align}
Using Lemma $\ref{Lem:Kolmo}$ and Lemma $\ref{Lem.E(xiNk)} $, we get
\begin{align}
\label{a33}
 \liminf_{n\to \infty} \frac{1}{2t_n}  \sum_{k=0}^{n-1} \xi_p(N_k) 
    & \stackrel{\text{a.s.}}{\ge} \liminf_{n\to \infty}  \frac{1}{2t_n}  \sum_{k=0}^{n-1} O(h^{\frac{3}{2}}) \dfrac{ 1 }{(1+ \Delta A_k )^3} \notag\\
     &  = Ch^{\frac{1}{2}} \left(  \liminf_{n\to \infty}   \frac{1}{2n}  \sum_{k=0}^{n-1} \frac{1}{(1+ \Delta A_k )^3} \right) = O(h^{\frac 12}). 
\end{align}
Thanks to $\eqref{a0'},$ $\eqref{a11},$ $\eqref{a22}$ and $\eqref{a33}$, we obtain
\begin{align*}
     \liminf_{n\to \infty} \frac{1}{t_n} \log |X_h(t_n)| &\stackrel{\text{a.s.}}{\ge}  \liminf_{n\to \infty} \left(  \frac{1}{n}\sum_{k=0}^{n-1} \dfrac{\log |1+ \Delta A_k |}{h}\right)   + \frac{\alpha^2}{2}\liminf_{n\to\infty}\left(\dfrac{1}{n}\sum_{k=0}^{n-1} \dfrac{1}{(1+ \Delta A_k )^2} \right) \notag \\
    & \quad\quad   - \alpha^2\limsup_{n\to\infty}\left(\dfrac{1}{n}\sum_{k=0}^{n-1} \dfrac{1}{(1+ \Delta A_k )^2} \right) + O(h^{\frac{1}{2}}). 
\end{align*}
This finishes the desired proof.
\end{proof}
\subsection{$p$-moment stability - Proof of Theorem \ref{Thm2}} 


\begin{lemma}
 \label{Lem.E(PsiNk)}  
For every $k\in \mathbb{N}$, it holds that
\begin{align*}
  &\mathbb{E} \left[\Psi_\rho \left(N_k \right) \right] =  O(h^{3/2}) \dfrac{1}{(1+\Delta A_k)^{3-p}}, 
\end{align*}
where 
\begin{align*}
    \Psi_\rho(N_k (\omega)):=  
    \left\{ \begin{array}{llrll} 
    \begin{pmatrix}
    p/2 \\
    4
    \end{pmatrix}  \dfrac{N_k^4}{(1+ \Delta A_k )^{8-p}} & \text{if}& \omega \in \Omega_{k,1},
    \vspace{0.15cm}\\
    \begin{pmatrix}
    p/2 \\
    3
    \end{pmatrix}  \dfrac{m ^{3-p/2} N_k^3}{(1+ \Delta A_k )^{6-p}}  & \text{if}& \omega \in \Omega_{k,2},
    \end{array} \right.
\end{align*}
where $\Omega_{k,1}$ and $\Omega_{k,2}$ are defined in \eqref{xirho}.
\end{lemma}
\begin{proof}
	Using the definition of $\mathcal H_k(y)$ in Lemma $\ref{Lem.E(xiNk)}$, we can rewrite 
	\begin{align*}
		\Psi_\rho(N_k) = (\Psi_\rho \circ \mathcal H_k ) \left(\dfrac{\Delta W_k}{\sqrt{h}} \right), \quad \forall k \ge 0.
	\end{align*}
	The expectation of $\Psi_\rho(N_k)$ can be estimated by 
	\begin{align*}
		&\mathbb{E} \left[\Psi_\rho \left(N_k \right) \right] \notag\\
		& =  \dfrac{p(2-p)(4-p)}{48} \times \dfrac{m^{3-p/2}}{(1+ \Delta A_k )^{6-p} \sqrt{2\pi}} \int_{\{ - \frac{m-1}{m} (1+ \Delta A_k )^2 < \mathcal H_k(y) < 0 \}} \Big( \mathcal H_k(y)\Big)^3  e^{-\frac{y^2}{2}}dy \notag\\
		& \quad - \dfrac{p(2-p)(4-p) (6-p)}{384} \times  \dfrac{1}{(1+ \Delta A_k )^{8-p} \sqrt{2\pi}} \int_{\{ \mathcal H_k(y) \ge 0 \}} \Big( \mathcal H_k(y)\Big)^4 e^{-\frac{y^2}{2}}dy\notag\\
		& = \dfrac{p(2-p)(4-p)}{48} \times \dfrac{m^{3-p/2}}{(1+ \Delta A_k )^{6-p} \sqrt{2\pi}} \int_{\{ - \frac{m-1}{m} (1+ \Delta A_k )^2 < \mathcal H_k(y) < 0 \}} \Bigg[ 2 \left( 1 + \Delta A_k \right) \left( \alpha \sqrt{h} y +\dfrac{\alpha^2h}{2} y^2\right)\notag\\
		& \quad  - \alpha^2h \left( \alpha \sqrt{h} y +\dfrac{\alpha^2h}{2} y^2 \right)^2  +  \left(\alpha \sqrt{h} y +\dfrac{\alpha^2h}{2} y^2\right)^2 +\left( \frac{\alpha^4}{4}h - \alpha^2  (1+  \Delta A_k ) \right) h \Bigg]^3  e^{-\frac{y^2}{2}}dy \notag\\
		& \quad - \dfrac{p(2-p)(4-p)(6-p)}{384} \times  \dfrac{1}{(1+ \Delta A_k )^{8-p} \sqrt{2\pi}} \int_{\{ \mathcal H_k(y) \ge 0 \}} \Bigg[ 2 \left( 1 + \Delta A_k \right) \left( \alpha \sqrt{h} y +\dfrac{\alpha^2h}{2} y^2\right)\notag\\
		& \quad - \alpha^2h \left( \alpha \sqrt{h} y +\dfrac{\alpha^2h}{2} y^2 \right)^2  +  \left(\alpha \sqrt{h} y +\dfrac{\alpha^2h}{2} y^2\right)^2 +\left( \frac{\alpha^4}{4}h - \alpha^2  (1+  \Delta A_k ) \right) h \Bigg]^4  e^{-\frac{y^2}{2}}dy \notag\\
		&= O(h^{3/2}) \dfrac{1}{(1+\Delta A_k)^{3-p}}.
	\end{align*}
	This finishes the proof.
\end{proof}
     
We are now in a position to prove Theorem $\ref{Thm2}$.
\begin{proof}[Proof of Theorem $\ref{Thm2}$]
It follows from \eqref{Nk} and \eqref{Gen} that
\begin{align}
	\label{Xp}
	\mathbb{E} |X_h(t_n)|^p 
	 = |x_0|^p \prod_{k=0}^{n-1} \mathbb{E} \, \Big|  1 + \Delta A_k -\dfrac{\alpha^2}{2} h +  {Y}_k   \Big|^p 
	 = |x_0|^p \prod_{k=0}^{n-1} \mathbb{E} \, \Big [  (1 + \Delta A_k)^2 + {N}_k   \Big]^{p/2}.
\end{align}
To calculate the expectations on the right-hand side, we will expand each product factor using the Taylor expansion given in Lemma \ref{Lem:TayTech} and truncating them at suitable orders. 

\medskip	
\textit{\underline{Upper estimate}}: Using inequality \eqref{Inq:Up}, we get
\begin{align*}
&\Big [(1 + \Delta A_k)^2 + {N}_k   \Big]^{p/2} \notag\\
&\le (1+ \Delta A_k)^p +\dfrac{p}{2}\dfrac{N_k}{(1+ \Delta A_k)^{2-p}}+\begin{pmatrix}
    p/2 \\
    2
    \end{pmatrix}\dfrac{N_k^2}{(1+ \Delta A_k)^{4-p}}  +\begin{pmatrix}
    p/2 \\
    3
    \end{pmatrix}\dfrac{N_k^3}{(1+ \Delta A_k)^{6-p}} .
    \end{align*}
Taking the expectation of two sides and applying Lemma \ref{Lem:ComputeYkj}, we have
\begin{align}
&\mathbb{E} \Big [(1 + \Delta A_k)^2 + {N}_k   \Big]^{p/2} \notag\\
& \le  (1+ \Delta A_k)^p + \dfrac{p}{2}\dfrac{1}{(1+ \Delta A_k)^{2-p}} \left( \alpha^2h+\dfrac{1}{2}\alpha^4h^2 \right) \notag\\
&\quad - \dfrac{1}{(1+ \Delta A_k)^{4-p}} \dfrac{p(2-p)}{8} \left(4(1+ \Delta A_k )^2 \left( \alpha^2 h + \dfrac{1}{2}\alpha^4h^2  \right) + 4 (1+ \Delta A_k) \beta_3 + O(h^2) \right) \notag\\
&\quad + \dfrac{1}{(1+ \Delta A_k)^{6-p}} \dfrac{p(2-p)(4-p)}{48} \left(8 (1+ \Delta A_k )^3 \beta_3 + 12 (1+ \Delta A_k )^2 \beta_4 
    + 6 (1+ \Delta A_k ) \beta_5 + O(h^2)\right) \notag\\
&  = (1+ \Delta A_k)^p - \dfrac{p(1-p)}{2} \dfrac{1}{(1+ \Delta A_k)^{2-p}} \alpha^2h + O(h^2).
\label{Np2}
\end{align}
Thanks to this and the expression of $ \mathbb{E} |X_h(t_n)|^p$ in \eqref{Xp} we get 
\begin{align}\mathbb{E} |X_h(t_n)|^p  \le |x_0|^p \prod_{k=0}^{n-1} \left((1+ \Delta A_k)^p - \dfrac{p(1-p)}{2} \dfrac{1}{(1+ \Delta A_k)^{2-p}} \alpha^2h + O(h^2)\right).
\label{expp}
\end{align}
Taking the logarithm of two sides of $\eqref{expp}$ and analyzing the limit superior, we obtain
\begin{align*}
   & \limsup_{n\to \infty}\dfrac{1}{t_n}\log \mathbb{E} |X_h(t_n)|^p \notag \\
   & \le  \limsup_{n\to \infty}\dfrac{1}{t_n} \left ( \log |x_0|^p + \sum_{k=0}^{n-1} \log \left((1+ \Delta A_k)^p - \dfrac{p(1-p)}{2} \dfrac{1}{(1+ \Delta A_k)^{2-p}} \alpha^2h + O(h^2) \right) \right) \notag\\
   & =  \limsup_{n\to \infty}\dfrac{1}{t_n} \left ( \sum_{k=0}^{n-1} \log \left((1+ \Delta A_k)^p - \dfrac{p(1-p)}{2} \dfrac{1}{(1+ \Delta A_k)^{2-p}} \alpha^2h + O(h^2) \right) \right)\notag\\
   & = \limsup_{n\to \infty}\dfrac{1}{t_n} \left ( \sum_{k=0}^{n-1} \log \left((1+ \Delta A_k)^p - p(1-p) \dfrac{1}{(1+ \Delta A_k)^{2-p}} \alpha^2h + O(h^2) \right. \right. \notag\\
   & \quad \quad \quad \quad \quad \quad \quad \quad \quad \quad\quad \left. \left. + \dfrac{p(1-p)}{2} \dfrac{1}{(1+ \Delta A_k)^{2-p}} \alpha^2h\right) \right).
   \end{align*} 
Applying inequality \eqref{Inq:Up} with $$\rho= (1+ \Delta A_k)^p - p(1-p) \dfrac{1}{(1+ \Delta A_k)^{2-p}} \alpha^2h + O(h^2)$$ and $$x= \dfrac{p(1-p)}{2} \dfrac{1}{(1+ \Delta A_k)^{2-p}} \alpha^2h,$$ we get
\begin{align*}
 &\log \left((1+ \Delta A_k)^p - p(1-p) \dfrac{1}{(1+ \Delta A_k)^{2-p}} \alpha^2h + O(h^2) + \dfrac{p(1-p)}{2} \dfrac{1}{(1+ \Delta A_k)^{2-p}} \alpha^2h\right) \notag\\
 &\leq \log \left( (1+ \Delta A_k)^p - p(1-p) \dfrac{1}{(1+ \Delta A_k)^{2-p}} \alpha^2h + O(h^2) \right) \notag\\
 & \quad + \dfrac{p(1-p) (1+ \Delta A_k)^{p-2} \alpha^2h}{2\left((1+ \Delta A_k)^p - p(1-p) (1+ \Delta A_k)^{p-2} \alpha^2h + O(h^2)\right)}\notag\\
  & \quad - \dfrac{p^2(1-p)^2 (1+ \Delta A_k)^{2p-4} \alpha^4h^2}{8\left( (1+ \Delta A_k)^p - p(1-p) (1+ \Delta A_k)^{p-2} \alpha^2h + O(h^2)\right )^2} \notag\\
  & \quad + \dfrac{p^3(1-p)^3 (1+ \Delta A_k)^{3p-6} \alpha^6h^3}{24\left( (1+ \Delta A_k)^p - p(1-p) (1+ \Delta A_k)^{p-2} \alpha^2h + O(h^2)\right )^3}.
\end{align*}
Thus, 
\begin{align*}
 & \limsup_{n\to \infty}\dfrac{1}{t_n}\log \mathbb{E} |X_h(t_n)|^p \notag \\
 &  \le \limsup_{n\to \infty}\dfrac{1}{t_n} \Bigg[ \sum_{k=0}^{n-1} \log \left( (1+ \Delta A_k)^p - p(1-p) \dfrac{1}{(1+ \Delta A_k)^{2-p}} \alpha^2h + O(h^2) \right) \notag\\
  & \qquad\qquad\qquad + \sum_{k=0}^{n-1} \dfrac{p(1-p) (1+ \Delta A_k)^{p-2} \alpha^2h}{2\left((1+ \Delta A_k)^p - p(1-p) (1+ \Delta A_k)^{p-2} \alpha^2h + O(h^2)\right)}\notag\\
  & \qquad\qquad\qquad  - \sum_{k=0}^{n-1} \dfrac{p^2(1-p)^2 (1+ \Delta A_k)^{2p-4} \alpha^4h^2}{8\left( (1+ \Delta A_k)^p - p(1-p) (1+ \Delta A_k)^{p-2} \alpha^2h + O(h^2)\right )^2} \notag\\
  & \qquad\qquad\qquad +  \sum_{k=0}^{n-1} \dfrac{p^3(1-p)^3 (1+ \Delta A_k)^{3p-6} \alpha^6h^3}{24\left( (1+ \Delta A_k)^p - p(1-p) (1+ \Delta A_k)^{p-2} \alpha^2h + O(h^2)\right )^3} \Bigg] \notag\\
   &  \le \limsup_{n\to \infty}\dfrac{1}{n} \Bigg[ \sum_{k=0}^{n-1} \dfrac{1}{h} \log \left( (1+ \Delta A_k)^p - p(1-p) \dfrac{1}{(1+ \Delta A_k)^{2-p}} \alpha^2h + O(h^2) \right) \notag\\
  & \qquad\qquad\qquad +  \sum_{k=0}^{n-1} \dfrac{p(1-p) (1+ \Delta A_k)^{p-2} \alpha^2 }{2\left((1+ \Delta A_k)^p - p(1-p) (1+ \Delta A_k)^{p-2} \alpha^2  + O(h^2)\right)} \Bigg] \notag\\
  & \qquad\qquad\qquad  - Ch\liminf_{n\to \infty}\dfrac{1}{n} \sum_{k=0}^{n-1} \dfrac{(1+ \Delta A_k)^{2p-4}}{8\left( (1+ \Delta A_k)^p - p(1-p) (1+ \Delta A_k)^{p-2} \alpha^2h + O(h^2)\right )^2} \notag\\
  & \qquad\qquad\qquad +  Ch^2\limsup_{n\to \infty}\dfrac{1}{n} \sum_{k=0}^{n-1} \dfrac{(1+ \Delta A_k)^{3p-6}}{24\left( (1+ \Delta A_k)^p - p(1-p) (1+ \Delta A_k)^{p-2} \alpha^2h + O(h^2)\right )^3}. 
\end{align*}
Thus, we obtain the desired upper bound
\begin{align*}
    \limsup_{n\to \infty}\dfrac{1}{t_n}\log \mathbb{E} |X_h(t_n)|^p  & \stackrel{\text{a.s.}} \le \limsup_{n\to \infty}\dfrac{1}{n} \Bigg[ \sum_{k=0}^{n-1}  \dfrac{1}{h} \log \left( (1+ \Delta A_k)^p - p(1-p) \dfrac{1}{(1+ \Delta A_k)^{2-p}} \alpha^2h + O(h^2) \right) \notag \\
    & \quad  +  \sum_{k=0}^{n-1} \dfrac{p(1-p) (1+ \Delta A_k)^{p-2} \alpha^2}{2\left((1+ \Delta A_k)^p - p(1-p) (1+ \Delta A_k)^{p-2} \alpha^2h + O(h^2)\right)}  \Bigg ] +O(h).
\end{align*}

\medskip
\noindent \textit{\underline{Lower estimate}}: 
Applying inequality in Remark \ref{remark:Taylor}, we have   
\begin{align}
\label{plower}
  \Big [(1 + \Delta A_k)^2 + {N}_k   \Big]^{p/2} \ge (1+ \Delta A_k)^p +\dfrac{p}{2}\dfrac{N_k}{(1+ \Delta A_k)^{2-p}}+\begin{pmatrix}
    p/2 \\
    2
   \end{pmatrix}\dfrac{N_k^2}{(1+ \Delta A_k)^{4-p}}  + \Psi_\rho(N_k),
\end{align}
where $\Psi_{\rho}(N_k)$ is defined in Lemma \ref{Lem.E(PsiNk)}.
Taking the expectation of two sides of \eqref{plower} and using Lemma \ref{Lem.E(PsiNk)}, we obtain
\begin{align}
&\mathbb{E} \Big [(1 + \Delta A_k)^2 + {N}_k   \Big]^{p/2} \notag\\
&\ge  (1+ \Delta A_k)^p + \dfrac{p}{2}\dfrac{1}{(1+ \Delta A_k)^{2-p}} \left( \alpha^2h+\dfrac{1}{2}\alpha^4h^2 \right)\notag\\
&\quad - \dfrac{1}{(1+ \Delta A_k)^{4-p}} \dfrac{p(2-p)}{8} \left(4(1+ \Delta A_k )^2 \left( \alpha^2 h + \dfrac{1}{2}\alpha^4h^2  \right) + 4 (1+ \Delta A_k) \beta_3 + O(h^2) \right) \notag\\
& \quad +O(h^{\frac{3}{2}}) \dfrac{1}{(1+\Delta A_k)^{3-p}} \notag\\ 
& = (1+ \Delta A_k)^p - \dfrac{p(1-p)}{2} \dfrac{1}{(1+ \Delta A_k)^{2-p}} \alpha^2h + O(h^{\frac{3}{2}}) \dfrac{1}{(1+\Delta A_k)^{3-p}}.
\label{Np3}
\end{align}
Thanks to the expression of $ \mathbb{E} |X_h(t_n)|^p$ and $\eqref{Np3},$ we get
\begin{align}\mathbb{E} |X_h(t_n)|^p  \ge |x_0|^p \prod_{k=0}^{n-1} \left((1+ \Delta A_k)^p - \dfrac{p(1-p)}{2} \dfrac{1}{(1+ \Delta A_k)^{2-p}} \alpha^2h + O(h^{\frac{3}{2}}) \dfrac{1}{(1+\Delta A_k)^{3-p}}\right).
\label{exppp}
\end{align}
Taking the logarithm of two sides of $\eqref{exppp}$, we have
\begin{align}
 &\log \mathbb{E} |X_h(t_n)|^p \notag\\
   & \ge  \log |x_0|^p + \sum_{k=0}^{n-1} \log \left((1+ \Delta A_k)^p - \dfrac{p(1-p)}{2} \dfrac{1}{(1+ \Delta A_k)^{2-p}} \alpha^2h + O(h^{\frac{3}{2}}) \dfrac{1}{(1+\Delta A_k)^{3-p}}   \right). 
   \label{upE}
\end{align}
We rewrite
\begin{align*}
&\log \left((1+ \Delta A_k)^p - \dfrac{p(1-p)}{2} \dfrac{1}{(1+ \Delta A_k)^{2-p}} \alpha^2h + O(h^{\frac{3}{2}}) \dfrac{1}{(1+\Delta A_k)^{3-p}} \right)\notag\\
    = &  \log \left((1+ \Delta A_k)^p - p(1-p) \dfrac{1}{(1+ \Delta A_k)^{2-p}} \alpha^2h + O(h^{\frac{3}{2}}) \dfrac{1}{(1+\Delta A_k)^{3-p}} + \dfrac{p(1-p)}{2} \dfrac{1}{(1+ \Delta A_k)^{2-p}} \alpha^2h\right). 
\end{align*}
Applying Lemma \ref{Lem:TayTech} with $$\rho= (1+ \Delta A_k)^p - p(1-p) \dfrac{1}{(1+ \Delta A_k)^{2-p}} \alpha^2h + O(h^{\frac{3}{2}}) \dfrac{1}{(1+\Delta A_k)^{3-p}} $$ and $$x= \dfrac{p(1-p)}{2} \dfrac{1}{(1+ \Delta A_k)^{2-p}} \alpha^2h \ge 0,$$ we get
\begin{align}
    & \log \left((1+ \Delta A_k)^p - p(1-p) \dfrac{1}{(1+ \Delta A_k)^{2-p}} \alpha^2h + O(h^{\frac{3}{2}})\dfrac{1}{(1+\Delta A_k)^{3-p}} + \dfrac{p(1-p)}{2} \dfrac{1}{(1+ \Delta A_k)^{2-p}} \alpha^2h\right) \notag\\
   & \ge  \log \left( (1+ \Delta A_k)^p - p(1-p) \dfrac{1}{(1+ \Delta A_k)^{2-p}} \alpha^2h + O(h^{\frac{3}{2}}) \dfrac{1}{(1+\Delta A_k)^{3-p}} \right)\notag \\
   & \qquad\quad + \dfrac{p(1-p) (1+ \Delta A_k)^{p-2} \alpha^2h}{2\left((1+ \Delta A_k)^p - p(1-p) (1+ \Delta A_k)^{p-2} \alpha^2h + O(h^{\frac{3}{2}}) (1+\Delta A_k)^{p-3} \right)} \notag \\
   & \qquad\quad - \dfrac{p^2(1-p)^2 (1+ \Delta A_k)^{2p-4} \alpha^4h^2}{8\left( (1+ \Delta A_k)^p - p(1-p) (1+ \Delta A_k)^{p-2} \alpha^2h + O(h^{\frac{3}{2}}) (1+\Delta A_k)^{p-3} \right )^2} \notag\\
   &\qquad\quad  + \dfrac{p^3(1-p)^3 (1+ \Delta A_k)^{3p-6} \alpha^6h^3}{24\left( (1+ \Delta A_k)^p - p(1-p) (1+ \Delta A_k)^{p-2} \alpha^2h + O(h^{\frac{3}{2}}) (1+\Delta A_k)^{p-3}\right )^3}
   \notag \\
    &\qquad\quad  - \dfrac{p^4(1-p)^4 (1+ \Delta A_k)^{4p-8} \alpha^8h^4}{64\left( (1+ \Delta A_k)^p - p(1-p) (1+ \Delta A_k)^{p-2} \alpha^2h + O(h^{\frac{3}{2}}) (1+\Delta A_k)^{p-3}\right )^4}\notag\\
      & =   \log \left( (1+ \Delta A_k)^p - p(1-p) \dfrac{1}{(1+ \Delta A_k)^{2-p}} \alpha^2h + O(h^{\frac{3}{2}}) \dfrac{1}{(1+\Delta A_k)^{3-p}} \right)\notag \\
   & \qquad\quad + \dfrac{p(1-p) (1+ \Delta A_k) \alpha^2h}{2\left((1+ \Delta A_k)^3 - p(1-p) (1+ \Delta A_k) \alpha^2h + O(h^{\frac{3}{2}})  \right)} \notag \\
   &\qquad \quad - \dfrac{p^2(1-p)^2 (1+ \Delta A_k)^2 \alpha^4h^2}{8\left( (1+ \Delta A_k)^3 - p(1-p) (1+ \Delta A_k) \alpha^2h + O(h^{\frac{3}{2}}) \right )^2} \notag\\
   &\qquad\quad  + \dfrac{p^3(1-p)^3 (1+ \Delta A_k)^3 \alpha^6h^3}{24\left( (1+ \Delta A_k)^3 - p(1-p) (1+ \Delta A_k) \alpha^2h + O(h^{\frac{3}{2}}) \right )^3}
   \notag \\
    &\qquad\quad  - \dfrac{p^4(1-p)^4 (1+ \Delta A_k)^4 \alpha^8h^4}{64\left( (1+ \Delta A_k)^3 - p(1-p) (1+ \Delta A_k) \alpha^2h + O(h^{\frac{3}{2}}) \right )^4}.
   \label{uplog}
\end{align}
Combining \eqref{upE} and \eqref{uplog}, we have 
\begin{align*}
& \log \mathbb{E} |X_h(t_n)|^p \notag \\
&  \ge \log |x_0|^p +\sum_{k=0}^{n-1} \Bigg[  \log \left( (1+ \Delta A_k)^p - p(1-p) \dfrac{1}{(1+ \Delta A_k)^{2-p}} \alpha^2h + O(h^{\frac{3}{2}}) \dfrac{1}{(1+\Delta A_k)^{3-p}} \right)\notag \\
   & \quad + \dfrac{p(1-p) (1+ \Delta A_k) \alpha^2h}{2\left((1+ \Delta A_k)^3 - p(1-p) (1+ \Delta A_k) \alpha^2h + O(h^{\frac{3}{2}})  \right)} \notag \\
   & \quad - \dfrac{p^2(1-p)^2 (1+ \Delta A_k)^2 \alpha^4h^2}{8\left( (1+ \Delta A_k)^3 - p(1-p) (1+ \Delta A_k) \alpha^2h + O(h^{\frac{3}{2}}) \right )^2} \notag\\
   &\quad  + \dfrac{p^3(1-p)^3 (1+ \Delta A_k)^3 \alpha^6h^3}{24\left( (1+ \Delta A_k)^3 - p(1-p) (1+ \Delta A_k) \alpha^2h + O(h^{\frac{3}{2}}) \right )^3}
   \notag \\
    &\quad  - \dfrac{p^4(1-p)^4 (1+ \Delta A_k)^4 \alpha^8h^4}{64\left( (1+ \Delta A_k)^3 - p(1-p) (1+ \Delta A_k) \alpha^2h + O(h^{\frac{3}{2}}) \right )^4} \Bigg].
\end{align*}
Now we can estimate using the sub-additivity of $\liminf$
\begin{align*}
  &  \liminf_{n\to \infty}\dfrac{1}{t_n}\log \mathbb{E} |X_h(t_n)|^p \notag \\
  & \ge  \liminf_{n\to \infty}\dfrac{1}{t_n} \Bigg[  \frac 1h\sum_{k=0}^{n-1}  \log \left( (1+ \Delta A_k)^p - p(1-p) \dfrac{1}{(1+ \Delta A_k)^{2-p}} \alpha^2h + O(h^{\frac{3}{2}})  \dfrac{1}{(1+\Delta A_k)^{3-p}} \right)\notag \\
   & \qquad\quad \qquad\quad +  \sum_{k=0}^{n-1}\dfrac{p(1-p)(1+ \Delta A_k)\alpha^2}{2\left((1+ \Delta A_k)^3 - p(1-p) (1+ \Delta A_k) \alpha^2h + O(h^{\frac{3}{2}})  \right)}\Bigg] \notag \\
   & \quad  - Ch\limsup_{n\to\infty}\frac{1}{n}\sum_{k=0}^{n-1}\dfrac{ (1+ \Delta A_k)^2}{8\left( (1+ \Delta A_k)^3 - p(1-p) (1+ \Delta A_k) \alpha^2h + O(h^{\frac{3}{2}}) \right )^2} \notag\\
   &\quad   + Ch^2\liminf_{n\to\infty}\frac{1}{n}\sum_{k=0}^{n-1}\dfrac{ (1+ \Delta A_k)^3}{24\left( (1+ \Delta A_k)^3 - p(1-p) (1+ \Delta A_k) \alpha^2h + O(h^{\frac{3}{2}}) \right )^3}
   \notag \\
    &\quad   -Ch^3\limsup_{n\to\infty0}\frac{1}{n}\sum_{k=0}^{n-1}\dfrac{ (1+ \Delta A_k)^4}{64\left( (1+ \Delta A_k)^3 - p(1-p) (1+ \Delta A_k) \alpha^2h + O(h^{\frac{3}{2}}) \right )^4}\Bigg]\notag\\
    & \ge  \liminf_{n\to \infty}\dfrac{1}{n} \Bigg[ \sum_{k=0}^{n-1}  \dfrac{1}{h} \log \left( (1+ \Delta A_k)^p - p(1-p) \dfrac{1}{(1+ \Delta A_k)^{2-p}} \alpha^2h + O(h^{\frac{3}{2}})  \dfrac{1}{(1+\Delta A_k)^{3-p}} \right)  \notag\\
   & \qquad\qquad\qquad +  \sum_{k=0}^{n-1} \dfrac{p(1-p) (1+ \Delta A_k) \alpha^2}{2\left((1+ \Delta A_k)^3 - p(1-p) (1+ \Delta A_k) \alpha^2h + O(h^{\frac{3}{2}})  \right)}  \Bigg] +O(h).
\end{align*}
This finishes the desired proof.
 \end{proof}

\section{Nonlinear non-autonomous equations}\label{sec:nonlinear}
Taking the logarithms of both sides of \eqref{Nonl}, we have
\begin{align*}
     \frac{1}{t_n}\log |X_h(t_n)| =& \frac{1}{t_n}\log |x_0| +  \frac{1}{t_n}\sum_{k=0}^{n-1} \log  
 \Bigg| 1 + \Delta A_k +\gamma_{1k} \Delta W_k + \dfrac{1}{2} \gamma_{2k} \left[(\Delta W_k)^2-h \right] \Bigg| \\
= &  \frac{1}{t_n}\log |x_0| +  \frac{1}{2t_n}\sum_{k=0}^{n-1} \log \Bigg( 1 + \Delta A_k +\gamma_{1k} \Delta W_k + \dfrac{1}{2} \gamma_{2k} \left[(\Delta W_k)^2-h \right]   \Bigg)^2  .
 \end{align*}
By denoting
\begin{align*}
\mathcal{Y}_k:= & \gamma_{1k} \Delta W_k + \dfrac{1}{2} \gamma_{2k} (\Delta W_k)^2,  \\
\mathcal{N}_k:= & 2(1 + \Delta A_k)\left(\mathcal{Y}_k -\dfrac{\gamma_{2k} h}{2} \right) + \left( \mathcal{Y}_k -\dfrac{\gamma_{2k} h}{2} \right)^2,
\end{align*}
we can rewrite
\begin{align*}
  \frac{1}{t_n}\log |X_h(t_n)| = \frac{1}{t_n}\log |x_0| +  \frac{1}{2t_n}\sum_{k=0}^{n-1} \log \left( (1+\Delta A_k)^2 + \mathcal{N}_k \right).  \end{align*}

\subsection{Lower bounds of $\mathcal{N}_k$ and expectation estimates}
\begin{lemma}
\label{Lem.Non.h}
 Assume \eqref{lowerbound_a} and \eqref{F}. Then, if $m>\frac{(4+4\gamma_*)^2}{(1+2\gamma_*)^2}$ then 
\begin{align*}
    \mathcal N_k  \stackrel{\text{a.s.}}{>} -\dfrac{m-1}{m} (1+\Delta A_k)^2 , \quad \forall  
     k \ge 0, 
\end{align*}
for any 
\begin{align*}
    0<h < h_{\ddagger}:= \dfrac{2(1-1/ \sqrt{m}) (1+\gamma_{\ast})-(\frac{3}{2}+\gamma_*)}{M_2}. 
\end{align*} 
\end{lemma}
\begin{proof}
Thanks to the definitions of $\gamma_{1k}$, $\gamma_{2k}$ and assumption \eqref{F}, we have
\begin{align*}
  \dfrac{ \gamma_{1k}^2}{\gamma_{2k}} = \left( \dfrac{f(t_n, X_h(t_k))}{X_h(t_k)} \right)^2 \dfrac{  X_h(t_k) }{   f(t_k, X_h(t_k)) \partial_xf(t_k, X_h(t_k))} = \dfrac{f(t_n, X_h(t_k))}{X_h(t_k) \partial_xf(t_k, X_h(t_k))} \leq \frac{3}{2}+\gamma_*.
\end{align*}
Thus,
\begin{align*}
    \dfrac{2(1-1/ \sqrt{m}) (1+\gamma_{\ast})-(\frac{3}{2}+\gamma_*)}{M_2}<\dfrac{2(1-1/ \sqrt{m}) (1+\gamma_{\ast})-\dfrac{ \gamma_{1k}^2}{\gamma_{2k}}}{\gamma_{2k}}
\end{align*} 
for any $k \ge 0.$
 Under \eqref{lowerbound_a}, we have $1+\gamma_\ast >1/2$. Therefore, if $m>\frac{(4+4\gamma_*)^2}{(1+2\gamma_*)^2}$ then $$2(1-1/ \sqrt{m}) (1+\gamma_{\ast})-\bigg(\frac{3}{2}+\gamma_*\bigg)>0,$$ which consequently guarantees that  $h_{\ddagger} >0$. Additionally, we also notice that from the definition of $\mathcal Y_k$ we have
 \[
    \mathcal Y_k = \frac 12\left(\Delta W_k \sqrt{\gamma_{2k}} + \frac{\gamma_{1k}}{\sqrt{\gamma_{2k}}}\right)^2 - \frac{\gamma_{1k}^2}{2\gamma_{2k}}\ge - \frac{\gamma_{1k}^2}{2\gamma_{2k}}
 \]
 almost surely for all $k\ge 0$. Hence,  
\begin{align*}
& \dfrac{2\sqrt{m}-2}{m} (1+\Delta A_k)^2 + \frac{2}{\sqrt{m}} (1+\Delta A_k) \left( {\mathcal Y}_k -\dfrac{\gamma_{2k}}{2} h  \right) \\
& \stackrel{\text{a.s.}}{\ge} \dfrac{2\sqrt{m}-2}{m} (1+\Delta A_k)^2 + \frac{2}{\sqrt{m}} (1+\Delta A_k ) \left(- \dfrac{\gamma_{1k}^2}{2 \gamma_{2k}}-\dfrac{\gamma_{2k}}{2} h \right)  \\
& \stackrel{\text{a.s.}}{\ge} \dfrac{1}{\sqrt{m}} (1+\Delta A_k) \left( 2\bra{1 - \frac{1}{\sqrt m}}(1+\Delta A_k) - \left( \dfrac{\gamma_{1k}^2}{ \gamma_{2k}}+ \gamma_{2k} h \right)  \right)\\
& \stackrel{\text{a.s.}}{\ge} \dfrac{1}{\sqrt{m}} (1+\gamma_\ast) \left( 2\bra{1 - \frac{1}{\sqrt m}}(1+\gamma_\ast) - \left( \dfrac{\gamma_{1k}^2}{ \gamma_{2k}}+ \gamma_{2k} h \right)  \right)
> 0, 
\end{align*}
for any $h$ such that $0<h\le h_{\ddagger}$. Using this estimate, for any integer $k\ge 0$ we now observe that  
\begin{align*}
\mathcal N_k + \dfrac{m-1}{m} (1+\Delta A_k)^2 &=   \left[ 2 \left( 1 + \Delta A_k \right) \left( {\mathcal Y}_k -\dfrac{\gamma_{2k}}{2} h  \right) +  \left( {\mathcal Y}_k -\dfrac{\gamma_{2k}}{2} h  \right)^2 \right] + \dfrac{m-1}{m} (1+\Delta A_k)^2 \\ 
&= \left(   \frac{\sqrt{m}-1}{\sqrt{m}}  \left( 1 + \Delta A_k \right) +  {\mathcal Y}_k -\dfrac{\gamma_{2k}}{2} h    \right)^2 + \dfrac{2\sqrt{m}-2}{m} (1+\Delta A_k)^2 \\
& \quad + \frac{2}{\sqrt{m}}  \left( 1 + \Delta A_k \right) \left( {\mathcal Y}_k -\dfrac{\gamma_{2k}}{2} h  \right) \\
& \stackrel{\text{a.s.}}{>} \dfrac{2\sqrt{m}-2}{m} (1+\Delta A_k)^2 + \frac{2}{\sqrt{m}}  \left( 1 + \Delta A_k \right) \left( {\mathcal Y}_k -\dfrac{\gamma_{2k}}{2} h  \right) \\
& \stackrel{\text{a.s.}}{>} 0,
\end{align*}
which shows the desired estimate for $\mathcal N_k$.    
\end{proof}

We proceed similarly with Lemma $\ref{Lem:ComputeYkj}$, we get

\begin{lemma}
 \label{Lem.Non.Nk}  
 The following estimates hold
\begin{align}
   \mathbb{E} (\mathcal{N}_k^l)  
    \le   \left\{
    \begin{array}{llll}
      f^{*2} h + \dfrac{1}{2}\fo^2h^2   & \text{if} \quad l=1, \\
    4(1+ \Delta A_k )^2 \left( f^{*2} h + \dfrac{1}{2}\fo^2h^2  \right) + 4 (1+ \Delta A_k) \alpha_3 + O(h^2)  & \text{if} \quad l=2, \\
    8 (1+ \Delta A_k )^3 \alpha_3 + 12 (1+ \Delta A_k )^2 \alpha_4 
    + 6 (1+ \Delta A_k ) \alpha_5 + O(h^2)   & \text{if} \quad l= 3,
    \end{array}
    \right.
    \label{en1}
\end{align}
and
\begin{align}
   \mathbb{E} (\mathcal{N}_k^l)  
    \ge   \left\{
    \begin{array}{llll}
      f^2_{*} h + \dfrac{1}{2}\fu^2h^2   & \text{if} \quad l=1, \\
    4(1+ \Delta A_k )^2 \left( f^2_{*} h + \dfrac{1}{2}\fu^2h^2  \right) + 4 (1+ \Delta A_k) \gamma_3 + O(h^2)  & \text{if} \quad l=2, \\
    8 (1+ \Delta A_k )^3 \gamma_3 + 12 (1+ \Delta A_k )^2 \gamma_4 
    + 6 (1+ \Delta A_k ) \gamma_5 + O(h^2)   & \text{if} \quad l= 3,
    \end{array}
    \right.
     \label{en2}
\end{align}
where $$\alpha_m:= \sup_{k \ge 0}\mathbb{E} \left(\mathcal{Y}_k -\dfrac{\gamma_{2k}}{2} h \right)^m \quad \text{ and } \quad \gamma_m:= \inf_{k \ge 0} \mathbb{E} \left(\mathcal{Y}_k -\dfrac{\gamma_{2k}}{2} h \right)^m \quad \text{for} \quad  m=3,4,5.$$
\end{lemma}
\begin{proof}
	Using Lemma \ref{Lem:BasicProp_0},  we compute directly   
	\begin{align*}
		\mathbb{E} ({\mathcal Y}_k^j) 
		&= \mathbb{E} \left[  \sum_{i=0}^j \begin{pmatrix}
			j \\
			i 
		\end{pmatrix} \big[ \gamma_{1k} \Delta W_k \big]^i \left( \dfrac{\gamma_{2k}}{2}(\Delta W_k)^2\right)^{j-i}   \right] = \mathbb{E} \left[  \sum_{i=0}^j \begin{pmatrix}
			j \\
			i 
		\end{pmatrix} \frac{ \gamma_{1k}^i \gamma_{2k}^{j-i}}{2^{j-i}} \Delta W_k^{2j-i}    \right]\\
		&= \mathbb{E} \left[  \sum_{i=0}^{\frac{1}{2}\jeven} \begin{pmatrix}
			j \\
			2i 
		\end{pmatrix}  \frac{\gamma_{1k}^{2i} \gamma_{2k}^{j-2i}}{2^{j-2i}} \Delta W_k^{2j-2i}    \right] =   \sum_{i=0}^{\frac{1}{2}\jeven} \begin{pmatrix}
			j \\
			2i 
		\end{pmatrix}  \frac{\gamma_{1k}^{2i} \gamma_{2k}^{j-2i}}{2^{j-2i}} [(2j-2i-1)!!] h^{j-i}    
	\end{align*}
	in which we note $j-i\ge j-\frac{1}{2}\jeven \ge 2$ for all $j\ge 3$.
	
	Moreover, we have
	\begin{align*}
		\mathbb{E} \left({\mathcal Y}_k -\dfrac{\gamma_{2k}}{2} h \right)^m & =  \mathbb{E} \left[  \sum_{j=0}^m \begin{pmatrix}
			m \\
			j 
		\end{pmatrix} {\mathcal Y}_k^j \left( -\dfrac{\gamma_{2k}}{2} h \right)^{m-j}   \right]\\
		& = \sum_{j=0}^m \begin{pmatrix}
			m \\
			j 
		\end{pmatrix}   \Bigg[ \sum_{i=0}^{\frac{1}{2}\jeven} \begin{pmatrix}
			j \\
			2i 
		\end{pmatrix} \dfrac{\gamma_{1k}^{2i} \gamma_{2k}^{j-2i}}{2^{j-2i}}  [(2j-2i-1)!!] h^{j-i}  \Bigg ] \left( -\dfrac{\gamma_{2k}}{2} h \right)^{m-j}\\ 
		& = \sum_{j=0}^m \begin{pmatrix}
			m \\
			j 
		\end{pmatrix}   \sum_{i=0}^{\frac{1}{2}\jeven} \begin{pmatrix}
			j \\
			2i 
		\end{pmatrix} (-1)^{m-j} \frac{\gamma_{1k}^{2i}  \gamma_{2k}^{m-2i}   }{2^{m-2i}} [(2j-2i-1)!!] h^{m-i}.
	\end{align*}
	Due to $f_* < \gamma_{1n} \leq f^*$ and $M_1 \le \gamma_{2n} \le M_2 $, we imply that $\mathbb{E} \left({\mathcal Y}_k -\dfrac{\gamma_{2k}}{2} h \right)^m $ is bounded for $m=3,4,5$. So, we can choose $\alpha_m$ and $\gamma_m$ such that 
    $$\alpha_m:= \sup_{k \ge 0}\mathbb{E} \left(\mathcal{Y}_k -\dfrac{\gamma_{2k}}{2} h \right)^m  = O(h^2) \quad \text{ and } \quad \gamma_m:= \inf_{k \ge 0} \mathbb{E} \left(\mathcal{Y}_k -\dfrac{\gamma_{2k}}{2} h \right)^m  = O(h^2) \quad \text{for} \quad  m=3,4,5.$$
	Thus,
	\begin{align*}
		\mathbb{E} \left({\mathcal Y}_k -\dfrac{\gamma_{2k}}{2} h \right)^m 
		=  \left\{
		\begin{array}{llll}
			0     & \text{if} \quad m=1, \\
			\gamma_{1k}^2 h  + \dfrac{1}{2} \gamma_{2k}^2 h^2  & \text{if} \quad m=2, \\
			O(h^2)   & \text{if} \quad m \ge 3.
		\end{array}
		\right.
	\end{align*}

	Thanks to the definition of $\mathcal N_k$, we get
	\begin{align}
		\label{k1}
		\mathbb{E} \left( \mathcal{N}_k \right) &=  \mathbb{E} \left( 2(1 + \Delta A_k)\left(\mathcal{Y}_k -\dfrac{\gamma_{2k} h}{2} \right) + \left( \mathcal{Y}_k -\dfrac{\gamma_{2k} h}{2} \right)^2 \right) \notag\\
		& = \gamma_{1k}^2 h  + \dfrac{1}{2} \gamma_{2k}^2 h^2 \leq f^{*2} h + \dfrac{1}{2}\fo^2h^2,
	\end{align}
	\begin{align}
		\label{k2}
		\mathbb{E} \left( \mathcal{N}_k \right)^2 &=  \mathbb{E} \left( 2(1 + \Delta A_k)\left(\mathcal{Y}_k -\dfrac{\gamma_{2k} h}{2} \right) + \left( \mathcal{Y}_k -\dfrac{\gamma_{2k} h}{2} \right)^2 \right)^2 \notag\\
		& = 4(1+\Delta A_k)^2 \left(\gamma_{1k}^2 h  + \dfrac{1}{2} \gamma_{2k}^2 h^2 \right) + 4(1+\Delta A_k)  \mathbb{E} \left({\mathcal Y}_k -\dfrac{\gamma_{2k}}{2} h \right)^3 + \mathbb{E} \left({\mathcal Y}_k -\dfrac{\gamma_{2k}}{2} h \right)^4 \notag\\
		& \leq 4(1+ \Delta A_k )^2 \left( f^{*2} h + \dfrac{1}{2}\fo^2h^2  \right) + 4 (1+ \Delta A_k) \alpha_3 + O(h^2)  ,
	\end{align}
	\begin{align}
		\label{k3}
		\mathbb{E} \left( \mathcal{N}_k \right)^3 &=  \mathbb{E} \left( 2(1 + \Delta A_k)\left(\mathcal{Y}_k -\dfrac{\gamma_{2k} h}{2} \right) + \left( \mathcal{Y}_k -\dfrac{\gamma_{2k} h}{2} \right)^2 \right)^3 \notag\\
		& = 8(1+\Delta A_k)^3 \mathbb{E} \left({\mathcal Y}_k -\dfrac{\gamma_{2k}}{2} h \right)^3 + 12(1+\Delta A_k)^2  \mathbb{E} \left({\mathcal Y}_k -\dfrac{\gamma_{2k}}{2} h \right)^4 \notag\\
		& \quad + 6 (1+\Delta A_k) \mathbb{E} \left({\mathcal Y}_k -\dfrac{\gamma_{2k}}{2} h \right)^5+\mathbb{E} \left({\mathcal Y}_k -\dfrac{\gamma_{2k}}{2} h \right)^6  \notag\\
		& \leq 8 (1+ \Delta A_k )^3 \alpha_3 + 12 (1+ \Delta A_k )^2 \alpha_4 
		+ 6 (1+ \Delta A_k ) \alpha_5 + O(h^2).
	\end{align}
	Combining $\eqref{k1},$ $\eqref{k2},$ $\eqref{k3},$ we prove $\eqref{en1}$. Similarly, we obtain $\eqref{en2}$. This finishes the proof.
\end{proof}

\subsection{Almost sure stability - Proof of Theorem \ref{Thm3}}
\begin{lemma}   
	\label{Lem.Non.xiNk}
	For every $k\in \mathbb{N}$, it holds that
	\begin{align*}
		\mathbb{E} \left[\xi_\rho \left( \mathcal{N}_k \right) \right] = O(h^{\frac{3}{2}})     \frac{1}{(1+ \Delta A_k )^3},  
	\end{align*}
	where $\xi_\rho$ is defined in \eqref{xirho}.
\end{lemma}
\begin{proof}
	We note that
	\begin{align*}
		\mathcal{N}_k &= 2 \left( 1 + \Delta A_k \right) \left( {Y}_k -\dfrac{\gamma_{2k}}{2} h  \right) +  \left( {Y}_k -\dfrac{\gamma_{2k}}{2} h  \right)^2\notag\\
		& = 2 \left( 1 + \Delta A_k \right) \left( \gamma_{1k}\Delta W_k +\dfrac{\gamma_{2k}}{2}(\Delta W_k)^2 -\dfrac{\gamma_{2k}}{2} h  \right) +  \left( \gamma_{1k}\Delta W_k +\dfrac{\gamma_{2k}}{2}(\Delta W_k)^2 -\dfrac{\gamma_{2k}}{2} h  \right)^2 \notag\\
		& = 2 \left( 1 + \Delta A_k \right) \left( \gamma_{1k}\Delta W_k +\dfrac{\gamma_{2k}}{2}(\Delta W_k)^2 \right) - \gamma_{2k} h \left(\gamma_{1k}\Delta W_k +\dfrac{\gamma_{2k}}{2}(\Delta W_k)^2 \right) \notag\\
		&  \quad +  \left(\gamma_{1k}\Delta W_k +\dfrac{\gamma_{2k}}{2}(\Delta W_k)^2 \right)^2 +\left( \frac{\gamma_{2k}^2}{4}h - \gamma_{2k}  (1+  \Delta A_k ) \right) h  \notag\\
		& =  2 \left( 1 + \Delta A_k \right) \left( \gamma_{1k} \sqrt{h} \left(\dfrac{\Delta W_k}{\sqrt{h}} \right)+\dfrac{\gamma_{2k} h}{2}\left(\dfrac{\Delta W_k}{\sqrt{h}}\right)^2\right) - \gamma_{2k} h \left( \gamma_{1k} \sqrt{h} \left(\dfrac{\Delta W_k}{\sqrt{h}} \right)+\dfrac{\gamma_{2k} h}{2}\left(\dfrac{\Delta W_k}{\sqrt{h}}\right)^2 \right)^2 \notag\\
		&  \quad +  \left( \gamma_{1k} \sqrt{h} \left(\dfrac{\Delta W_k}{\sqrt{h}} \right)+\dfrac{\gamma_{2k} h}{2}\left(\dfrac{\Delta W_k}{\sqrt{h}}\right)^2\right)^2 +\left( \frac{\gamma_{2k}^2}{4}h - \gamma_{2k}  (1+  \Delta A_k ) \right) h.
	\end{align*}
	We can define the function
	\begin{align}
		\label{G_k}
		\mathcal G_k(y) &=  2 \left( 1 + \Delta A_k \right) \left( \gamma_{1k} \sqrt{h} y +\dfrac{\gamma_{2k} h}{2} y^2\right) - \gamma_{2k} h \left( \gamma_{1k} \sqrt{h} y +\dfrac{\gamma_{2k} h}{2} y^2 \right)^2 \notag\\
		&  \quad +  \left(\gamma_{1k} \sqrt{h} y +\dfrac{\gamma_{2k} h}{2} y^2\right)^2 +\left( \frac{\gamma_{2k}^2}{4}h - \gamma_{2k}  (1+ \Delta A_k ) \right) h.   
	\end{align}
	We can write briefly
	\begin{align*}
		\mathcal N_k(\omega) = \mathcal G_k \left(\dfrac{\Delta W_k(\omega)}{\sqrt{h}} \right), \quad \forall  \omega \in \Omega,  k \ge 0.
	\end{align*}
	Thanks to Lemma \ref{Lem.Non.h}, we see that $\mathcal G_k (\Delta W_k(\omega)/\sqrt{h})$ takes values at least $-((m-1)/m)(1+\Delta A_k)^2$ almost surely. Hence, the process $\xi_\rho(\mathcal N_k)$ can be rewritten as 
	\begin{align*}
		\xi_\rho(\mathcal N_k) = (\xi_\rho \circ \mathcal G_k ) \left(\dfrac{\Delta W_k}{\sqrt{h}} \right), \quad \forall k \ge 0.
	\end{align*}
	To estimate the expectation of the process $\mathcal N_k$, in the following we will use the probability density function of $\Delta W_k/\sqrt{h}$. Indeed, we have 
	\begin{align*}
		& \mathbb{E} \left[\xi_\rho \left( \mathcal{N}_k \right) \right]\\
         &  = \dfrac{m^3}{3(1+ \Delta A_k )^6 \sqrt{2\pi}} \int_{\{ - \frac{m-1}{m} (1+ \Delta A_k )^2 < \mathcal G_k(y) < 0 \}}  \Big( \mathcal{G}_k(y)  \Big)^3  e^{-\frac{y^2}{2}}dy \notag\\
		& \quad - \dfrac{1}{4(1+ \Delta A_k )^8 \sqrt{2\pi}} \int_{\{ \mathcal G_k(y) \ge 0 \}} \Big( \mathcal{G}_k(y)  \Big)^4   e^{-\frac{y^2}{2}}dy  \notag\\ 
		& = \dfrac{m^3}{3(1+ \Delta A_k )^6 \sqrt{2\pi}} \int_{\{ - \frac{m-1}{m} (1+ \Delta A_k )^2 < \mathcal G_k(y) < 0 \}} \Bigg[ 2 \left( 1 + \Delta A_k \right) \left( \gamma_{1k}  \sqrt{h} y +\dfrac{\gamma_{2k} h}{2} y^2\right)\notag\\
		& \quad - \gamma_{2k} h \left( \gamma_{1k} \sqrt{h} y +\dfrac{\gamma_{2k} h}{2} y^2 \right)^2   +  \left(\gamma_{1k} \sqrt{h} y +\dfrac{\gamma_{2k} h}{2} y^2\right)^2 +\left( \frac{\gamma_{2k}^2}{4}h - \gamma_{2k}  (1+  \Delta A_k ) \right) h \Bigg]^3  e^{-\frac{y^2}{2}}dy \notag\\
		&\quad - \dfrac{1}{4(1+ \Delta A_k )^8 \sqrt{2\pi}} \int_{\{ \mathcal G_k(y) \ge 0 \}}  \Bigg[ 2 \left( 1 + \Delta A_k \right) \left( \gamma_{1k} \sqrt{h} y +\dfrac{\gamma_{2k} h}{2} y^2\right)\notag\\
		&\quad  - \gamma_{2k} h \left( \gamma_{1k} \sqrt{h} y +\dfrac{\gamma_{2k} h}{2} y^2 \right)^2  +  \left(\gamma_{1k} \sqrt{h} y +\dfrac{\gamma_{2k} h}{2} y^2\right)^2 +\left( \frac{\gamma_{2k}^2}{4}h - \gamma_{2k}  (1+  \Delta A_k ) \right) h \Bigg]^4  e^{-\frac{y^2}{2}}dy  \notag\\
		& = O(h^{\frac{3}{2}})     \frac{1}{(1+ \Delta A_k )^3}.
	\end{align*}
	This finishes the proof.
\end{proof}

We now give a proof of Theorem \ref{Thm3}.
\begin{proof}[Proof of Theorem $\ref{Thm3}$]\hfill
	
\noindent \textit{\underline{Upper estimate}}:  Using inequality \eqref{Inq:Up}, we get
\begin{align}
\log \left( (1+ \Delta A_k )^2 + \mathcal{N}_k \right) \le \log (1+ \Delta A_k )^2 + \frac{\mathcal{N}_k}{(1+ \Delta A_k )^2}  -  \frac{\mathcal{N}_k^2}{2(1+ \Delta A_k )^4}   + \frac{\mathcal{N}_k^3}{3(1+ \Delta A_k )^6} . 
\label{b0}
\end{align}
Using Lemma \ref{Lem:Kolmo}, \eqref{en1}, and \eqref{en2}, we have
 \begin{align}
      \limsup_{n\to \infty} \frac{1}{2t_n}  \sum_{k=0}^{n-1} \frac{\mathcal{N}_k}{(1+ \Delta A_k )^2}
      &\stackrel{\text{a.s.}}{\leq} \limsup_{n\to \infty} \left( \frac{f^{*2}}{2n}  \sum_{k=0}^{n-1} \frac{1}{(1+ \Delta A_k )^2}   + \frac{\fo^2h}{4n}  \sum_{k=0}^{n-1} \frac{1}{(1+ \Delta A_k )^2} \right)\notag \\
      & \le \frac{f^{*2}}{2}\limsup_{n\to \infty} \left( \frac{1}{n}  \sum_{k=0}^{n-1} \frac{1}{(1+ \Delta A_k )^2}\right)   + h\limsup_{n\to \infty}\left(\frac{\fo^2}{4n}  \sum_{k=0}^{n-1} \frac{1}{(1+ \Delta A_k )^2} \right)\notag \\
      & \le \frac{f^{*2}}{2}\limsup_{n\to \infty} \left( \frac{1}{n}  \sum_{k=0}^{n-1} \frac{1}{(1+ \Delta A_k )^2}\right)   + O(h).
      \label{b1}
\end{align}
For the higher order term w.t.r $\mathcal N_k$, we use $\gamma_3 = O(h^2)$ and $\alpha_m = O(h^2)$ for $m\in \{3,4,5\}$ to estimate
\begin{align}
        \liminf_{n\to \infty}  \frac{1}{4t_n}  \sum_{k=0}^{n-1} \frac{\mathcal{N}_k^2}{(1+ \Delta A_k )^4} &\stackrel{\text{a.s.}}{\geq}  \liminf_{n\to \infty}  \frac{1}{t_n}  \sum_{k=0}^{n-1} \frac{(1+ \Delta A_k )^2 ( f_{*}^2 h +  \frac{1}{2}\fu^2 h^2  ) + (1+ \Delta A_k) \gamma_3 + O(h^2) }{(1+ \Delta A_k )^4} \notag\\  
      & \stackrel{\text{a.s.}}{\geq} f_*^2\liminf_{n\to \infty} \left( \frac{1}{n}  \sum_{k=0}^{n-1} \frac{1}{(1+ \Delta A_k )^2}\right) + O(h), 
      \label{b2}
\end{align}
and
\begin{align}
\label{b3}
& \limsup_{n\to \infty} \frac{1}{6t_n}  \sum_{k=0}^{n-1} \frac{\mathcal{N}_k^3}{(1+ \Delta A_k )^6} \notag\\
    & \stackrel{\text{a.s.}}{\le} \limsup_{n\to \infty}  \frac{1}{6t_n}  \sum_{k=0}^{n-1} \frac{  8 (1+ \Delta A_k )^3 \alpha_3 + 12 (1+ \Delta A_k )^2 \alpha_4 
     + 6 (1+ \Delta A_k ) \alpha_5 + O(h^2) }{(1+ \Delta A_k )^6} 
   = O(h).
\end{align}
Analyzing the limit superior, and combining $\eqref{b0}$, $\eqref{b1}, \eqref{b2}, \eqref{b3}$, we obtain the desired upper bound
\begin{align*}
    \limsup_{n\to \infty} \frac{1}{t_n} \log |X_h(t_n)| 
    & \stackrel{\text{\normalfont a.s.}}{\le} \limsup_{n\to \infty} \left(  \frac{1}{n}\sum_{k=0}^{n-1} \dfrac{\log |1+ \Delta A_k |}{h}\right)  + \frac{f^{*2}}{2}\limsup_{n\to\infty}\frac{1}{n}\sum_{k=0}^{n-1}\frac{1}{(1+\Delta A_k)^2}\\
  &  \quad  - f_*^2\liminf_{n\to\infty}\frac{1}{n}\sum_{k=0}^{n-1}\frac{1}{(1+\Delta A_k)^2} + O(h).   
\end{align*}
\noindent \textit{\underline{Lower estimate}}: 
We note that \begin{align*}
\mathcal{Y}_k = \gamma_{1k}\Delta W_k + \dfrac{\gamma_{2k}}{2} (\Delta W_k)^2  = \left(\sqrt{\dfrac{\gamma_{2k}}{2}} \Delta W_k + \dfrac{ \gamma_{1k}}{\sqrt{2} \sqrt{\gamma_{2k}}} \right)^2 - \dfrac{ \gamma_{1k}^2}{2 \gamma_{2k}} \ge - \dfrac{\gamma_{1k}^2}{2 \gamma_{2k}}. 
\end{align*}
We can choose $h$ sufficiently small such that $1+\Delta A_k$ is strictly positive. By Lemma $\ref{Lem.Non.h}$, we can choose 
$$0< h < \dfrac{2(1-1/ \sqrt{m}) (1+\gamma_{\ast})-(\frac{3}{2}+\gamma_*)}{M_2}  $$ 
and the inequality $ \mathcal{N}_k > -\dfrac{m-1}{m} (1+\Delta A_k)^2$ is sastisfied almost surely. 
Applying inequality in Remark \ref{remark:Taylor}, we have
\begin{align*}
\log \left( (1+ \Delta A_k )^2 + \mathcal{N}_k \right) \ge \log (1+ \Delta A_k )^2 + \frac{\mathcal{N}_k}{(1+ \Delta A_k )^2}  -  \frac{\mathcal{N}_k^2}{2(1+ \Delta A_k )^4} + \xi_\rho(\mathcal{N}_k).
\end{align*}
Using Lemma $\ref{Lem:Kolmo}$, $\eqref{en1},$ and $\eqref{en2}$, we have
 \begin{align}
      \liminf_{n\to \infty} \frac{1}{2t_n}  \sum_{k=0}^{n-1} \frac{\mathcal{N}_k}{(1+ \Delta A_k )^2}
      &\stackrel{\text{a.s.}}{\ge} \liminf_{n\to \infty} \left( \frac{f_{*}^2}{2n}  \sum_{k=0}^{n-1} \frac{1}{(1+ \Delta A_k )^2}   + \frac{\fu^2h}{4n}  \sum_{k=0}^{n-1} \frac{1}{(1+ \Delta A_k )^2} \right) \notag \\
      & =\frac{f_*^2}{2}\liminf_{n\to\infty}\frac{1}{n}\sum_{k=0}^{n-1}\frac{1}{(1+\Delta A_k)^2} + O(h).
      \label{b11}
\end{align}
For the term concerning $\mathcal N_k^2$, we use $\alpha_3 = O(h^2)$ to estimate
\begin{align}
        \limsup_{n\to \infty} \frac{1}{4t_n}  \sum_{k=0}^{n-1} \frac{\mathcal{N}_k^2}{(1+ \Delta A_k )^4} &\stackrel{\text{a.s.}}{\leq}  \limsup_{n\to \infty}  \frac{1}{t_n}  \sum_{k=0}^{n-1} \frac{(1+ \Delta A_k )^2 ( f^{*2} h +  \frac{1}{2}\fo^2 h^2  ) + (1+ \Delta A_k) \alpha_3 + O(h^2) }{(1+ \Delta A_k )^4} \notag\\  
      & \stackrel{\text{a.s.}}{=}   f^{*2}\limsup_{n\to \infty}  \frac{1}{n}  \sum_{k=0}^{n-1} \frac{1}{(1+ \Delta A_k )^2} + O(h). 
      \label{b22}
\end{align}
Using Lemma $\ref{Lem:Kolmo}$ and Lemma $\ref{Lem.Non.xiNk} $, we get
\begin{align}
\label{b33}
 \liminf_{n\to \infty} \frac{1}{2t_n}  \sum_{k=0}^{n-1} \xi_p(\mathcal{N}_k) 
    & \stackrel{\text{a.s.}}{\ge} \liminf_{n\to \infty}  \frac{1}{2t_n}  \sum_{k=0}^{n-1} O(h^{\frac{3}{2}}) \dfrac{ 1 }{(1+ \Delta A_k )^3}   = O(h^{\frac{1}{2}}). 
\end{align}
By combining $\eqref{b11}, \eqref{b22}, \eqref{b33}$, we obtain
\begin{align*}
    \liminf_{n\to \infty} \frac{1}{t_n} \log |X_h(t_n)| 
   &\ge  \liminf_{n\to \infty}  \frac{1}{2t_n} \sum_{k=0}^{n-1}\left[ \log (1+ \Delta A_k )^2 + \frac{ \mathcal{N}_k}{(1+ \Delta A_k )^2}  -  \frac{\mathcal{N}_k^2}{2(1+ \Delta A_k )^4} + \xi_\rho(\mathcal{N}_k) \right] \notag \\
    & \stackrel{\text{a.s.}}{\ge}  \liminf_{n\to \infty}   \frac{1}{n}\sum_{k=0}^{n-1} \dfrac{\log |1+ \Delta A_k |}{h}  + \frac{f^2_{*}}{2}\liminf_{n\to\infty}\dfrac{1}{n}  \sum_{k=0}^{n-1} \dfrac{1}{(1+ \Delta A_k )^2} \notag \\
    & \quad   - f^{*2}\limsup_{n\to\infty}\dfrac{1}{n}  \sum_{k=0}^{n-1} \dfrac{1}{(1+ \Delta A_k )^2} + O(h^{\frac{1}{2}}). 
\end{align*}
This finishes the desired proof.
\end{proof}

\subsection{$p$-moment stability - Proof of Theorem \ref{Thm4}}
\begin{lemma}
    \label{Lem.Non.PsiNk}
    For every $k\in \mathbb{N}$, it holds that
    \begin{align*}
\mathbb{E} \left[\Psi_\rho \left(\mathcal{N}_k \right) \right]  =  O(h^{3/2}) \dfrac{1}{(1+\Delta A_k)^{3-p}},  
    \end{align*}
    where $\Psi_\rho$ is defined in Lemma $\ref{Lem.E(PsiNk)}$.
\end{lemma}
\begin{proof}
	Using the definition of $\mathcal G_k(y)$ in equation $\eqref{G_k}$ in Lemma $\ref{Lem.Non.xiNk}$, we can rewrite 
	\begin{align*}
		\Psi_\rho(\mathcal N_k) = (\Psi_\rho \circ \mathcal G_k ) \left(\dfrac{\Delta W_k}{\sqrt{h}} \right), \quad \forall k \ge 0.
	\end{align*}
	The expectation of $\Psi_\rho(\mathcal{N}_k)$ is given by 
	\begin{align*}
	&	\mathbb{E} \left[\Psi_\rho\left(\mathcal{N}_k \right) \right] \notag\\
		& =  \dfrac{p(p-2)(p-4)}{48} \times \dfrac{m^{3-p/2}}{(1+ \Delta A_k )^{6-p} \sqrt{2\pi}} \int_{\{ - \frac{m-1}{m} (1+ \Delta A_k )^2 < \mathcal G_k(y) < 0 \}} \Big( \mathcal{G}_k(y)\Big)^3  e^{-\frac{y^2}{2}}dy \notag\\
		& \quad +\dfrac{p(p-2)(p-4) (p-6)}{384} \times  \dfrac{1}{(1+ \Delta A_k )^{8-p} \sqrt{2\pi}} \int_{\{ \mathcal G_k(y) \ge 0 \}} \Big( \mathcal{G}_k(y)\Big)^4 e^{-\frac{y^2}{2}}dy\notag\\
		& = \dfrac{p(p-2)(p-4)}{48} \times \dfrac{m^{3-p/2}}{(1+ \Delta A_k )^{6-p} \sqrt{2\pi}} \notag\\
		& \quad \times \int_{\{ - \frac{m-1}{m} (1+ \Delta A_k )^2 < \mathcal G_k(y) < 0 \}}  \Bigg[ 2 \left( 1 + \Delta A_k \right) \left( \gamma_{1k} \sqrt{h} y +\dfrac{\gamma_{2k}h}{2} y^2\right)\notag\\
		& \quad    - \gamma_{2k} h \left( \gamma_{1k} \sqrt{h} y +\dfrac{\gamma_{2k} h}{2} y^2 \right)^2  +  \left(\gamma_{1k} \sqrt{h} y +\dfrac{\gamma_{2k} h}{2} y^2\right)^2   +\left( \frac{\gamma_{2k}^2}{4}h - \gamma_{2k}  (1+  \Delta A_k ) \right) h \Bigg]^3  e^{-\frac{y^2}{2}}dy \notag\\
		& \quad +\dfrac{p(p-2)(p-4)(p-6)}{384} \times  \dfrac{1}{(1+ \Delta A_k )^{8-p} \sqrt{2\pi}} \notag\\
		& \quad \times \int_{\{ \mathcal G_k(y) \ge 0 \}} \Bigg[ 2 \left( 1 + \Delta A_k \right) \left( \gamma_{1k} \sqrt{h} y +\dfrac{\gamma_{2k}h}{2} y^2\right)\notag\\
		& \quad   - \gamma_{2k} h \left( \gamma_{1k} \sqrt{h} y +\dfrac{\gamma_{2k}h}{2} y^2 \right)^2  +  \left(\gamma_{1k} \sqrt{h} y +\dfrac{\gamma_{2k}h}{2} y^2\right)^2  +\left( \frac{\gamma_{2k}^2}{4}h - \gamma_{2k}  (1+  \Delta A_k ) \right) h \Bigg]^4  e^{-\frac{y^2}{2}}dy \notag\\
		& = O(h^{3/2}) \dfrac{1}{(1+\Delta A_k)^{3-p}}.
	\end{align*} 
	This finishes the proof.
\end{proof}

We are now in a position to prove Theorem \ref{Thm4}.
\begin{proof}
It follows from \eqref{Nonl} that
\begin{align*}
    \mathbb{E} |X_h(t_n)|^p 
     = |x_0|^p \prod_{k=0}^{n-1} \mathbb{E} \, \Big|  1 + \Delta A_k -\dfrac{\gamma_{2k}}{2} h +  \mathcal{Y}_k   \Big|^p = |x_0|^p \prod_{k=0}^{n-1} \mathbb{E} \, \Big [  (1 + \Delta A_k)^2 + \mathcal{N}_k   \Big]^{p/2}.
    \end{align*}
\noindent \textit{\underline{Upper estimate}}:  Using inequality \eqref{Inq:Up}, we get
\begin{align}
&\Big [(1 + \Delta A_k)^2 + \mathcal{N}_k   \Big]^{p/2} \notag\\
&\le (1+ \Delta A_k)^p +\dfrac{p}{2}\dfrac{\mathcal{N}_k}{(1+ \Delta A_k)^{2-p}}+\begin{pmatrix}
    p/2 \\
    2
    \end{pmatrix}\dfrac{\mathcal{N}_k^2}{(1+ \Delta A_k)^{4-p}}  +\begin{pmatrix}
    p/2 \\
    3
    \end{pmatrix}\dfrac{\mathcal{N}_k^3}{(1+ \Delta A_k)^{6-p}}.
    \label{p1}
    \end{align}
Taking the expectation of two sides of $\eqref{p1}$, and combining $\eqref{en1}, \eqref{en2},$ we have
\begin{align*}
&\mathbb{E} \Big [(1 + \Delta A_k)^2 + {\mathcal N}_k   \Big]^{p/2}\\
&\le  (1+ \Delta A_k)^p + \dfrac{p}{2}\dfrac{1}{(1+ \Delta A_k)^{2-p}} \left( f^{*2}h+\dfrac{1}{2} \fo^2 h^2 \right)\\
&\quad - \dfrac{1}{(1+ \Delta A_k)^{4-p}} \dfrac{p(2-p)}{8} \left(4(1+ \Delta A_k )^2 \left( f^2_{*} h + \dfrac{1}{2} \fu^2 h^2  \right) + 4 (1+ \Delta A_k) \alpha_3 + O(h^2) \right)\\
&\quad + \dfrac{1}{(1+ \Delta A_k)^{6-p}} \dfrac{p(2-p)(4-p)}{48} \left(8 (1+ \Delta A_k )^3 \alpha_3 + 12 (1+ \Delta A_k )^2 \alpha_4 
    + 6 (1+ \Delta A_k ) \alpha_5 + O(h^2)\right)\\
& = (1+ \Delta A_k)^p + \dfrac{p}{2} \dfrac{1}{(1+ \Delta A_k)^{2-p}} \left(f^{*2}-(2-p)f_*^2 \right) h + O(h^2).
\end{align*}
Thus, $$\mathbb{E} |X_h(t_n)|^p \le |x_0|^p \prod_{k=0}^{n-1} \left((1+ \Delta A_k)^p + \dfrac{p}{2} \dfrac{1}{(1+ \Delta A_k)^{2-p}} \left(f^{*2} - (2-p)f_*^2 \right) h  + O(h^2)\right). $$
Taking logarithms on both sides and then taking the limit superior, we obtain
\begin{align*}
   & \limsup_{n\to \infty}\dfrac{1}{t_n}\log \mathbb{E} |X_h(t_n)|^p \notag \\
   & \le  \limsup_{n\to \infty}\dfrac{1}{t_n} \left ( \log |x_0|^p + \sum_{k=0}^{n-1} \log \left((1+ \Delta A_k)^p + \dfrac{p}{2} \dfrac{1}{(1+ \Delta A_k)^{2-p}} \left(f^{*2}-(2-p)f_*^2 \right) h  + O(h^2) \right) \right) \notag\\
   & =  \limsup_{n\to \infty}\dfrac{1}{t_n} \left ( \sum_{k=0}^{n-1} \log \left((1+ \Delta A_k)^p + \dfrac{p}{2} \dfrac{1}{(1+ \Delta A_k)^{2-p}} \left(f^{*2}-(2-p)f_*^2 \right) h  + O(h^2) \right) \right)\notag\\
   & = \limsup_{n\to \infty}\dfrac{1}{t_n} \left ( \sum_{k=0}^{n-1} \log \left((1+ \Delta A_k)^p + \dfrac{p}{2} \dfrac{1}{(1+ \Delta A_k)^{2-p}} (f^{*2} - (3-2p)f^2_{*}) h + O(h^2) \right. \right. \notag\\
   & \quad \quad \quad \quad \quad \quad \quad \quad \quad \quad\quad \left. \left. + \dfrac{p(1-p)}{2} \dfrac{1}{(1+ \Delta A_k)^{2-p}} f^2_{*} h\right) \right).
   \end{align*} 
Applying inequality \eqref{Inq:Up} with $$\rho= (1+ \Delta A_k)^p + \dfrac{p}{2} \dfrac{1}{(1+ \Delta A_k)^{2-p}} (f^{*2} - (3-2p)f^2_{*}) h + O(h^2)$$ and $$x= \dfrac{p(1-p)}{2} \dfrac{1}{(1+ \Delta A_k)^{2-p}} f^2_{*} h,$$ we get
\begin{align*}
 &\log \left((1+ \Delta A_k)^p + \dfrac{p}{2} \dfrac{1}{(1+ \Delta A_k)^{2-p}} (f^{*2} - (3-2p)f^2_{*}) h + O(h^2) + \dfrac{p(p-2)}{2} \dfrac{1}{(1+ \Delta A_k)^{2-p}} f^2_{*} h\right) \notag\\
 &\leq \log \left( (1+ \Delta A_k)^p + \dfrac{p}{2} \dfrac{1}{(1+ \Delta A_k)^{2-p}} (f^{*2} - (3-2p)f^2_{*}) h + O(h^2) \right) \notag\\
 & \quad + \dfrac{p(1-p)  f^2_{*} h}{2(1+ \Delta A_k)^2 + p  (f^{*2} - (3-2p)f^2_{*}) h +2(1+ \Delta A_k)^{2-p}  O(h^2)}\notag\\
 & \quad - \dfrac{p^2(1-p)^2 f^4_{*} h^2}{2\left(2(1+ \Delta A_k)^2 + p  (f^{*2} - (3-2p)f^2_{*}) h +2(1+ \Delta A_k)^{2-p}  O(h^2)\right)^2}\notag\\
 & \quad + \dfrac{p^3(1-p)^3 f^6_{*} h^3}{3\left(2(1+ \Delta A_k)^2 + p  (f^{*2} - (3-2p)f^2_{*}) h +2(1+ \Delta A_k)^{2-p}  O(h^2)\right)^3}.
\end{align*}
Thus, 
\begin{align*}
 & \limsup_{n\to \infty}\dfrac{1}{t_n}\log \mathbb{E} |X_h(t_n)|^p  \\
 & \le \limsup_{n\to \infty}\dfrac{1}{t_n} \Bigg[ \sum_{k=0}^{n-1} \log  \left( (1+ \Delta A_k)^p + \dfrac{p}{2} \dfrac{1}{(1+ \Delta A_k)^{2-p}} (f^{*2} - (3-2p)f^2_{*}) h + O(h^2) \right)  \notag\\
  & \qquad \qquad \quad + \sum_{k=0}^{n-1} \dfrac{p(1-p) f^2_{*} h}{2(1+ \Delta A_k)^2 + p  (f^{*2} - (3-2p)f^2_{*}) h +2(1+ \Delta A_k)^{2-p}  O(h^2)}\notag\\
  & \qquad \qquad \quad  - \sum_{k=0}^{n-1} \dfrac{p^2(1-p)^2 f^4_{*} h^2}{2\left(2(1+ \Delta A_k)^2 + p  (f^{*2} - (3-2p)f^2_{*}) h +2(1+ \Delta A_k)^{2-p}  O(h^2)\right)^2}\notag\\
  & \qquad \qquad \quad +  \sum_{k=0}^{n-1} \dfrac{p^3(1-p)^3 f^6_{*} h^3}{3\left(2(1+ \Delta A_k)^2 + p  (f^{*2} - (3-2p)f^2_{*}) h +2(1+ \Delta A_k)^{2-p}  O(h^2)\right)^3} \Bigg] \notag\\
   & \le \limsup_{n\to \infty}\dfrac{1}{n} \Bigg[ \sum_{k=0}^{n-1} \dfrac{1}{h}\log \left( (1+ \Delta A_k)^p + \dfrac{p}{2} \dfrac{1}{(1+ \Delta A_k)^{2-p}} (f^{*2} - (3-2p)f^2_{*}) h + O(h^2) \right)  \notag\\
  & \qquad \qquad \quad +  \sum_{k=0}^{n-1} \dfrac{p(1-p) f^2_{*} }{2(1+ \Delta A_k)^2 + p  (f^{*2} - (3-2p)f^2_{*}) h +2(1+ \Delta A_k)^{2-p}  O(h^2)} \Bigg] \notag\\
  & \qquad \qquad \quad - Ch\liminf_{n\to \infty}\dfrac{1}{n} \sum_{k=0}^{n-1}  \dfrac{1}{2\left(2(1+ \Delta A_k)^2 + p  (f^{*2} - (3-2p)f^2_{*}) h +2(1+ \Delta A_k)^{2-p}  O(h^2)\right)^2}\notag\\
  & \qquad \qquad \quad  +Ch^2 \limsup_{n\to \infty}\dfrac{1}{n} \sum_{k=0}^{n-1} \dfrac{1}{3\left(2(1+ \Delta A_k)^2 + p  (f^{*2} - (3-2p)f^2_{*}) h +2(1+ \Delta A_k)^{2-p}  O(h^2)\right)^3}. 
\end{align*}
Therefore, we have
\begin{align*}
    &\limsup_{n\to \infty}\dfrac{1}{t_n}\log \mathbb{E} |X_h(t_n)|^p\\
    &\le \limsup_{n\to \infty}\dfrac{1}{n} \Bigg[ \sum_{k=0}^{n-1}  \dfrac{1}{h} \log \left( (1+ \Delta A_k)^p + \dfrac{p}{2} \dfrac{1}{(1+ \Delta A_k)^{2-p}} (f^{*2} - (3-2p)f^2_{*}) h + O(h^2) \right) \notag \\
    &\qquad \qquad +  \sum_{k=0}^{n-1} \dfrac{p(1-p) f^2_{*} }{2(1+ \Delta A_k)^2 + p  (f^{*2} - (3-2p)f^2_{*}) h +2(1+ \Delta A_k)^{2-p}  O(h^2)}  \Bigg ] +O(h).
\end{align*}

\medskip
\noindent \textit{\underline{Lower estimate}}: 
Applying inequality in Remark \ref{remark:Taylor}, we have    
\begin{align}
\label{plower2}
  \Big [(1 + \Delta A_k)^2 + \mathcal{N}_k   \Big]^{p/2} \ge (1+ \Delta A_k)^p +\dfrac{p}{2}\dfrac{\mathcal{N}_k}{(1+ \Delta A_k)^{2-p}}+\begin{pmatrix}
    p/2 \\
    2
    \end{pmatrix}\dfrac{\mathcal{N}_k^2}{(1+ \Delta A_k)^{4-p}}  + \Psi_\rho(\mathcal{N}_k).
\end{align}
Taking the expectation of two sides of \eqref{plower2}, and combining $\eqref{en1}, \eqref{en2},$ we have
\begin{align}
&\mathbb{E} \Big [(1 + \Delta A_k)^2 + \mathcal{N}_k   \Big]^{p/2} \notag\\
&\ge  (1+ \Delta A_k)^p + \dfrac{p}{2}\dfrac{1}{(1+ \Delta A_k)^{2-p}} \left( f_*^2h+\dfrac{1}{2}\fu^2 h^2 \right)\notag\\
&\quad - \dfrac{1}{(1+ \Delta A_k)^{4-p}} \dfrac{p(2-p)}{8} \left(4(1+ \Delta A_k )^2 \left( f^{*2} h + \dfrac{1}{2}\fo^2 h^2  \right) + 4 (1+ \Delta A_k) \gamma_3 + O(h^2) \right) \notag\\
& \quad + O(h^{\frac{3}{2}}) \dfrac{1}{(1+\Delta A_k)^{3-p}} \notag\\ 
& = (1+ \Delta A_k)^p + \dfrac{p}{2} \dfrac{1}{(1+ \Delta A_k)^{2-p}} (f_*^2 - (3-2p)f^{*2}) h +  \dfrac{p(1-p)}{2} \dfrac{1}{(1+ \Delta A_k)^{2-p}}f^{*2}  h \notag \\
& \quad + O(h^{\frac{3}{2}})  \dfrac{1}{(1+\Delta A_k)^{3-p}}.
\label{p2}
\end{align}
Taking the logarithm of two sides of $\eqref{p2}$, we get
\begin{align}
 \log \mathbb{E} |X_h(t_n)|^p  & = \log \left[ |x_0|^p \prod_{k=0}^{n-1} \mathbb{E} \, \Big [  (1 + \Delta A_k)^2 + \mathcal{N}_k   \Big]^{p/2} \right] \notag\\
   & \ge  \log |x_0|^p + \sum_{k=0}^{n-1} \log \left((1+ \Delta A_k)^p + \dfrac{p}{2} \dfrac{1}{(1+ \Delta A_k)^{2-p}} (f_*^2 - (3-2p)f^{*2}) h \right. \notag\\
   & \quad + \left. \dfrac{p(1-p)}{2} \dfrac{1}{(1+ \Delta A_k)^{2-p}}f^{*2} h + O(h^{\frac{3}{2}})  \dfrac{1}{(1+\Delta A_k)^{3-p}} \right). 
   \label{upE2}
\end{align}
Applying inequality in Lemma \ref{Lem:TayTech} with $$\rho= (1+ \Delta A_k)^p + \dfrac{p}{2} \dfrac{1}{(1+ \Delta A_k)^{2-p}} (f_*^2 - (3-2p)f^{*2}) h + O(h^{\frac{3}{2}})  \dfrac{1}{(1+\Delta A_k)^{3-p}} $$ and $$x= \dfrac{p(1-p)}{2} \dfrac{1}{(1+ \Delta A_k)^{2-p}}f^{*2} h \ge 0,$$ we get
\begin{align}
    & \log \left((1+ \Delta A_k)^p + \dfrac{p}{2} \dfrac{1}{(1+ \Delta A_k)^{2-p}} (f_*^2 - (3-2p)f^{*2}) h + O(h^{\frac{3}{2}})  \dfrac{1}{(1+\Delta A_k)^{3-p}} \right. \notag\\
    & \quad \left. +  \dfrac{p(1-p)}{2} \dfrac{1}{(1+ \Delta A_k)^{2-p}}f^{*2} h  \right) \notag\\
   & \ge  \log \left( (1+ \Delta A_k)^p + \dfrac{p}{2} \dfrac{1}{(1+ \Delta A_k)^{2-p}} (f_*^2 - (3-2p)f^{*2}) h + O(h^{\frac{3}{2}}) \dfrac{1}{(1+\Delta A_k)^{3-p}} \right) \notag \\
   & \quad + \dfrac{p(1-p) (1+\Delta A_k) f^{*2} h}{2(1+ \Delta A_k)^3 + p (1+ \Delta A_k)  (f_*^2 - (3-2p)f^{*2}) h + O(h^{\frac{3}{2}})}\notag\\ 
   & \quad - \dfrac{p^2(1-p)^2  (1+ \Delta A_k)^2 f^{*4} h^2}{2\left(2(1+ \Delta A_k)^3 + p (1+ \Delta A_k) (f_*^2 - (3-2p)f^{*2}) h + O(h^{\frac{3}{2}})\right)^2}\notag\\ 
   &\quad  + \dfrac{p^3(1-p)^3 (1+ \Delta A_k)^3 f^{*6} h^3}{3\left(2(1+ \Delta A_k)^3 + p (1+ \Delta A_k) (f_*^2 - (3-2p)f^{*2}) h +  O(h^{\frac{3}{2}})\right)^3}\notag\\ 
    &\quad  - \dfrac{p^4(1-p)^4 (1+ \Delta A_k)^4 f^{*8} h^4}{4\left(2(1+ \Delta A_k)^3 + p (1+ \Delta A_k) (f_*^2 - (3-2p)f^{*2}) h +   O(h^{\frac{3}{2}})\right)^4}.
   \label{uplog2}
\end{align}
Thanks to \eqref{upE2} and \eqref{uplog2}, we have 
\begin{align*}
&\log \mathbb{E} |X_h(t_n)|^p \notag\\ &  \ge  \log |x_0|^p + \sum_{k=0}^{n-1} \Bigg[  \log \left( (1+ \Delta A_k)^p + \dfrac{p}{2} \dfrac{1}{(1+ \Delta A_k)^{2-p}} (f_*^2 - (3-2p)f^{*2}) h + O(h^{\frac{3}{2}}) \dfrac{1}{(1+ \Delta A_k)^{3-p}} \right)\notag \\
  & \qquad\qquad\qquad\quad\quad + \dfrac{p(1-p) (1+ \Delta A_k) f^{*2} h}{2(1+ \Delta A_k)^3 + p (1+ \Delta A_k) (f_*^2 - (3-2p)f^{*2}) h +  O(h^{\frac{3}{2}})}\notag\\ 
   & \qquad\qquad\qquad\quad\quad - \dfrac{p^2(1-p)^2 (1+ \Delta A_k)^2 f^{*4} h^2}{2\left(2(1+ \Delta A_k)^3 + p (1+ \Delta A_k) (f_*^2 - (3-2p)f^{*2}) h +  O(h^{\frac{3}{2}})\right)^2} \notag\\ 
   &\qquad\qquad\qquad\quad\quad  + \dfrac{p^3(1-p)^3 (1+ \Delta A_k)^3 f^{*6} h^3}{3\left(2(1+ \Delta A_k)^3 + p (1+ \Delta A_k) (f_*^2 - (3-2p)f^{*2}) h +  O(h^{\frac{3}{2}})\right)^3}\notag\\ 
    &\qquad\qquad\qquad\quad\quad  - \dfrac{p^4(1-p)^4 (1+ \Delta A_k)^4 f^{*6} h^3}{4\left(2(1+ \Delta A_k)^3 + p (1+ \Delta A_k) (f_*^2 - (3-2p)f^{*2}) h +  O(h^{\frac{3}{2}})\right)^4} \Bigg].
\end{align*}
Thus,
\begin{align*}
  &  \liminf_{n\to \infty}\dfrac{1}{t_n}\log \mathbb{E} |X_h(t_n)|^p \notag \\
   & \ge  \liminf_{n\to \infty}\dfrac{1}{t_n}  \sum_{k=0}^{n-1} \Bigg[  \log \left( (1+ \Delta A_k)^p + \dfrac{p}{2} \dfrac{1}{(1+ \Delta A_k)^{2-p}} (f_*^2 - (3-2p)f^{*2}) h + O(h^{\frac{3}{2}}) \dfrac{1}{(1+ \Delta A_k)^{3-p}}\right)\notag \\
   & \quad\quad\quad\quad\quad\quad\quad   + \dfrac{p(1-p) (1+ \Delta A_k) f^{*2} h}{2(1+ \Delta A_k)^3 + p (1+ \Delta A_k) (f_*^2 - (3-2p)f^{*2}) h +  O(h^{\frac{3}{2}})}\notag\\ 
   & \quad\quad\quad\quad\quad\quad\quad - \dfrac{p^2(1-p)^2 (1+ \Delta A_k)^2 f^{*4} h^2}{2\left(2(1+ \Delta A_k)^3 + p (1+ \Delta A_k) (f_*^2 - (3-2p)f^{*2}) h +   O(h^{\frac{3}{2}})\right)^2}\notag\\ 
   &\quad\quad\quad\quad\quad\quad\quad  + \dfrac{p^3(1-p)^3 (1+ \Delta A_k)^3 f^{*6} h^3}{3\left(2(1+ \Delta A_k)^3 + p (1+ \Delta A_k) (f_*^2 - (3-2p)f^{*2}) h +  O(h^{\frac{3}{2}})\right)^3}\notag\\ 
    &\quad\quad\quad\quad\quad\quad\quad  - \dfrac{p^4(1-p)^4 (1+ \Delta A_k)^4 f^{*6} h^4}{4\left(3(1+ \Delta A_k)^3 + p (1+ \Delta A_k) (f_*^2 - (3-2p)f^{*2}) h +  O(h^{\frac{3}{2}})\right)^4} \Bigg] 
    \notag \\
    & \ge  \liminf_{n\to \infty} \dfrac{1}{n}\Bigg[  \sum_{k=0}^{n-1} \frac 1h\log \left( (1+ \Delta A_k)^p + \dfrac{p}{2} \dfrac{1}{(1+ \Delta A_k)^{2-p}} (f_*^2 - (3-2p)f^{*2}) h + O(h^{\frac{3}{2}}) \dfrac{1}{(1+ \Delta A_k)^{3-p}}\right)\notag \\
   & \quad\quad\quad\quad\quad\quad\quad    + \sum_{k=0}^{n-1}\dfrac{p(1-p)(1+ \Delta A_k)f^{*2}}{2(1+ \Delta A_k)^3 + p (1+ \Delta A_k) (f_*^2 - (3-2p)f^{*2}) h +  O(h^{\frac{3}{2}})}\Bigg]\notag\\ 
   & \quad - Ch\limsup_{n\to\infty}\frac{1}{n}\sum_{k=0}^{n-1}\dfrac{(1+ \Delta A_k)^2}{2\left(2(1+ \Delta A_k)^3 + p (1+ \Delta A_k) (f_*^2 - (3-2p)f^{*2}) h +   O(h^{\frac{3}{2}})\right)^2} \notag\\ 
   &\quad   + Ch^2\liminf_{n\to\infty}\frac{1}{n}\sum_{k=0}^{n-1}\dfrac{(1+ \Delta A_k)^3}{3\left(2(1+ \Delta A_k)^3 + p (1+ \Delta A_k) (f_*^2 - (3-2p)f^{*2}) h +  O(h^{\frac{3}{2}})\right)^3}\notag\\ 
    &\quad   -Ch^3\limsup_{n\to\infty}\frac{1}{n}\sum_{k=0}^{n-1} \dfrac{(1+ \Delta A_k)^4}{4\left(3(1+ \Delta A_k)^3 + p (1+ \Delta A_k) (f_*^2 - (3-2p)f^{*2}) h +  O(h^{\frac{3}{2}})\right)^4} \\
    & =  \liminf_{n\to \infty} \frac{1}{n}\sum_{k=0}^{n-1}\dfrac{1}{n} \Bigg[ \sum_{k=0}^{n-1}  \dfrac{1}{h} \log \left( (1+ \Delta A_k)^p + \dfrac{p}{2} \dfrac{1}{(1+ \Delta A_k)^{2-p}} (f_*^2 - (3-2p)f^{*2}) h + O(h^{\frac{3}{2}}) \dfrac{1}{(1+\Delta A_k)^{3-p}} \right)  \notag\\
   & \quad \quad  \quad \quad  \quad \quad  +  \sum_{k=0}^{n-1} \dfrac{p(1-p) (1+\Delta A_k) f^{*2} }{2(1+ \Delta A_k)^3 + p (1+\Delta A_k) (f_*^2 - (3-2p)f^{*2}) h +  O(h^{\frac{3}{2}})}  \Bigg] +O(h).
\end{align*}

This finishes the desired proof.
\end{proof}

\section{Numerical simulations}\label{sec:num}
In this section, we present numerical simulations illustrating the results obtained in the previous sections. 

\medskip
We start with the linear non-autonomous equation
\begin{equation*}
    dX(t) = a(t)X(t)dt + \alpha X(t) dW(t)
\end{equation*}
where the function $a(t)$ is bounded and the Lyapunov exponent $\limsup_{t\to\infty}\frac 1t\int_0^ta(s)ds$ is positive. More precisely, we take three examples 
\[
    a(t) \in \left\{\frac{4t}{t + 4}, \frac{5t}{t + 5}, \frac{6t}{t+6}\right\}
\]
where the Lyapunov exponents are $4, 5$ and $6$ respectively. That means that the corresponding deterministic equations are (exponentially) unstable. We solve the SDE for different values of $\alpha\in \{2, 4\}$ with the final time horizon $T = 3$ and the step size $h = 10^{-3}$. We calculate 50 samples and plot the mean of $\log|X_h(t)|$ as the bold line, while the variance arising from different samples is represented by the region around it with the same color but lighter. When $\alpha = 2$, which is small, we see in the Figure \ref{fig:1a} that the solution is still unstable since the stochastic Lyapunov exponent \eqref{Lyapunov_exp} is still positive. In Figure \ref{fig:1b} we see that with enough noise intensities, the solution decays exponentially to zero, which is represented by $\log|X_h(t)|$ which behaves like linear functions with negative slopes.


\begin{figure}[H]
\begin{center}
\begin{subfigure}{.45\linewidth}
\centering
\includegraphics[width=8.4cm, height=6.5cm]{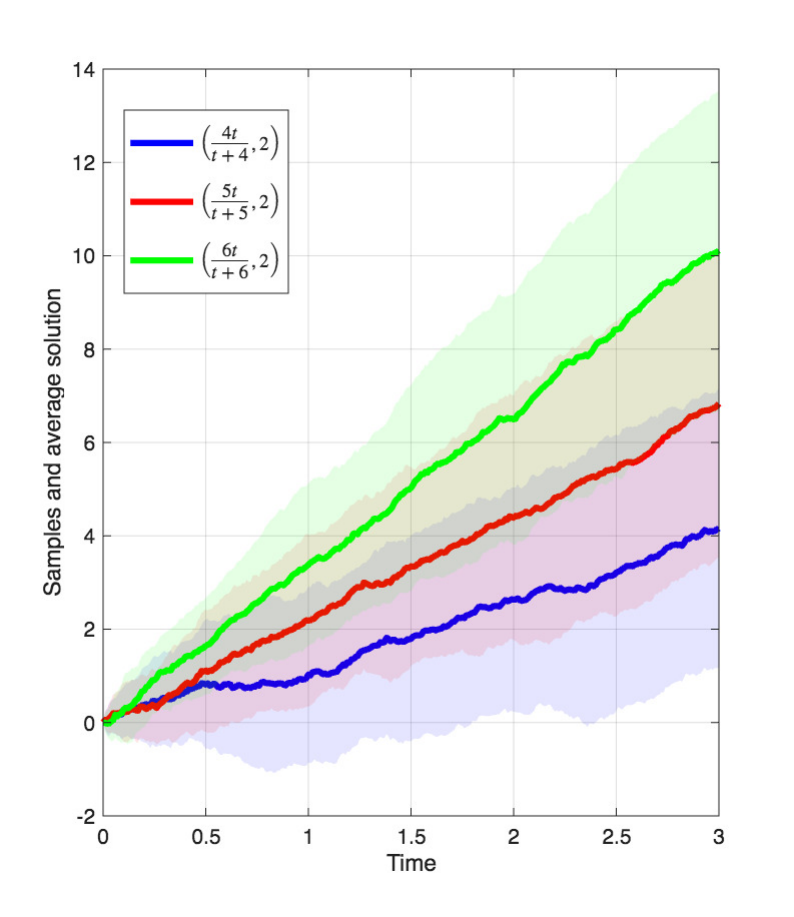}

\caption{Small noise intensities are insufficient to stabilise the unstable deterministic problem.}
\label{fig:1a}
\end{subfigure}%
\hspace*{0.2in}
\begin{subfigure}{.45\linewidth}
\centering
\includegraphics[width=8.4cm, height=6.5cm]{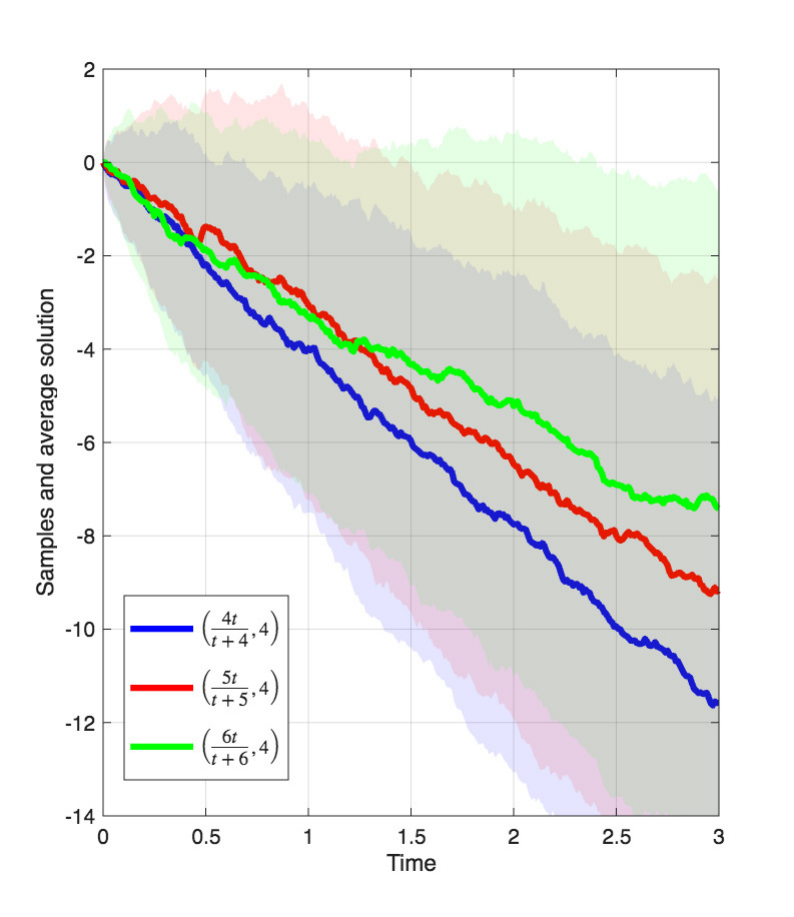}
\caption{Sufficiently large noise intensities stabilise the unstable deterministic problem.}
\label{fig:1b}
\end{subfigure}%
\caption{Effect of different noise intensities on unstable deterministic linear non-autonomous problems.}
\label{fig:1}
\end{center}
\end{figure}

In the second example, we consider the case where the non-autonomous coefficient $a(t)$ is unbounded and still satisfies the finiteness of the Lyapunov exponent \eqref{finiteness_a}. More precisely, 
\begin{equation}\label{discontinuousa}
    a(t)= \begin{cases}
n+1, & t\in \left[n,n+\frac{1}{n}\right],\\
0, & t\in \left(n+\frac{1}{n},\,n+1\right],
\end{cases}
\end{equation}
which gives the Lyapunov exponent of the deterministic equation
\begin{equation*}
    \limsup_{t\to\infty}\int_0^ta(s)ds = 1.
\end{equation*}
This means that the deterministic equation is exponentially stable. It follows that for any $|\alpha| >  \sqrt{2}$, the SDE $dX(t) = a(t)X(t)dt + \alpha X(t)dW(t)$ is almost surely exponentially stable. This can be seen in Figure \ref{fig:2}, with the time horizon $T = 15$ and time step size $h = 10^{-3}$, where the averages of logarithmic of solutions with different $\alpha$ are presented. When $\alpha=0$, i.e. the deterministic equation, the solution is exponentially unstable (black line). For $\alpha = 1 < \sqrt{2}$, the solution is still unstable, though the growth rate is lower than the deterministic one. For $\alpha = 2$ and $\alpha = 3$, the logarithmic of solution behaves like a linear function with 
negative slope, meaning that the solution is exponentially stable, and the large $|\alpha|$ is the faster the decay to zero holds.
\begin{figure}[H]
    \centering
    \includegraphics[width=0.5\linewidth]{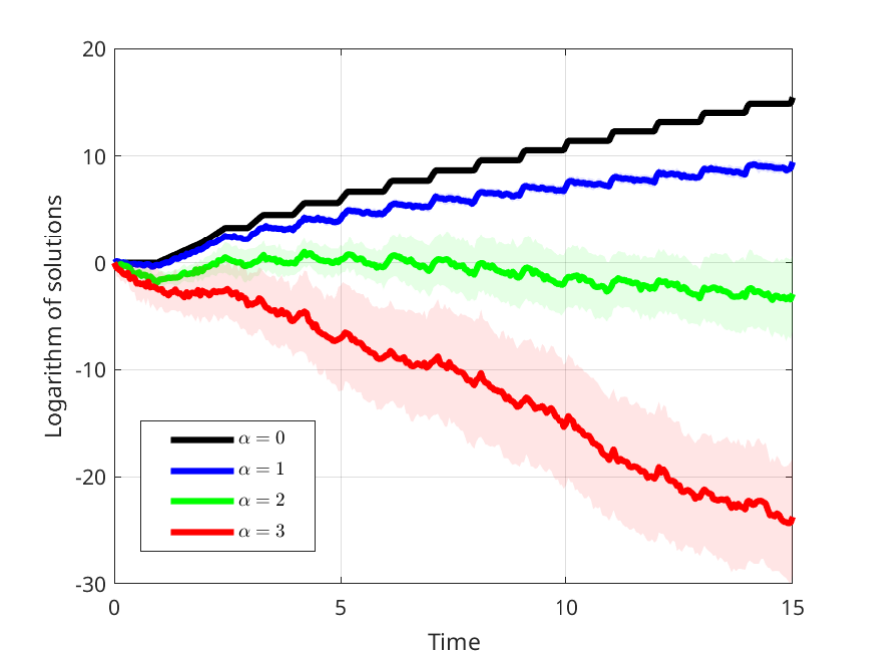}
    \caption{Stabilisation by noise for the Milstein scheme of the SDE with the discontinuous and unbounded coefficient $a(t)$ in \eqref{discontinuousa}.}
    \label{fig:2}
\end{figure}


Finally, we consider the case of nonlinear equation \eqref{SDE:Continuous0_non} with
\begin{equation*}
    a(t) = 3 + \sin(t)^2 \quad\text{ and } \quad f(t,X) = \alpha\times\frac{(2+t^2)(1+x^2)x}{(1+t^2)(2+x^2)}.
\end{equation*}
It is clear that for $\alpha = 0$, the deterministic equation is unstable, which is presented by the black line in Figure \ref{fig:3}.  When $\alpha$ increases, the stabilising effect of the nonlinear term also increases which is represented for $\alpha = 3, 6$ and $8$, respectively, in Figure \ref{fig:3}.
\begin{figure}[H]
    \centering
    \includegraphics[width=0.5\linewidth]{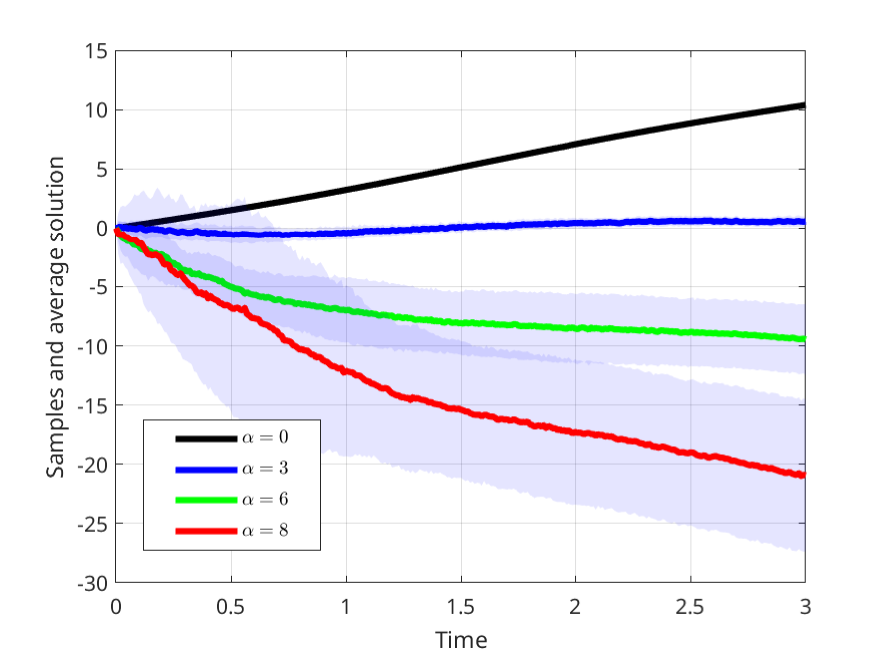}
    \caption{The stabilising effect of the noise for the nonlinear non-autonomous SDE.}
    \label{fig:3}
\end{figure}

\section{Appendix}
\addtocontents{toc}{\setcounter{tocdepth}{1}}

\subsection{Proof of Lemma \ref{Lem:TayTech}.} \label{Apd.Lem:TayTech}
\begin{proof}
We consider 
\begin{align*}
    h(x) &= f_{\eta,r}(x) - f_{\eta,r}^{\mathsf{low}}(x) \notag\\
    & = \left\{ \begin{array}{llrll}
    (\eta+x)^r - \eta^r -  \dfrac{1}{\eta^{1-r}} r x - 
    \dfrac{1}{\eta^{2-r}} \begin{pmatrix}
    r \\
    2
    \end{pmatrix}  x^2 - \dfrac{1}{\eta^{3-r}} \begin{pmatrix}
    r \\
    3 
    \end{pmatrix} 
     x^3 - \dfrac{1}{\eta^{4-r}} \begin{pmatrix}
    r \\
    4 
    \end{pmatrix} 
     x^4 \qquad  \text{if } x \ge 0,
    \vspace{0.15cm}\\
    (\eta+x)^r - \eta^r -  \dfrac{1}{\eta^{1-r}} r x - 
    \dfrac{1}{\eta^{2-r}} \begin{pmatrix}
    r \\
    2
    \end{pmatrix}  x^2 - \left( \dfrac{m} {\eta}\right)^{3-r} \begin{pmatrix}
    r \\
    3 
    \end{pmatrix} 
     x^3 \quad  \text{if } - \dfrac{m-1}{m}\eta < x < 0. 
    \end{array} \right.
\end{align*}
The case $x=0$ is trivial. 

We compute the derivatives of $h(x)$ and analyze their signs.
\begin{align*}
    h'(x)= \left\{ \begin{array}{llrll}
    r(\eta+x)^{r-1} -   \dfrac{1}{\eta^{1-r}} r  - 
    \dfrac{r(r-1)}{\eta^{2-r}} x - \dfrac{r(r-1)(r-2)}{2\eta^{3-r}}  
     x^2 - \dfrac{r(r-1)(r-2)(r-3)}{6\eta^{4-r}} x^3 \qquad  \text{if } x > 0,
    \vspace{0.15cm}\\
    r(\eta+x)^{r-1} -  \dfrac{1}{\eta^{1-r}} r  - 
    \dfrac{r(r-1)}{\eta^{2-r}}  x - \left( \dfrac{m} {\eta}\right)^{3-r} \dfrac{r(r-1)(r-2)}{2} x^2
    \quad  \text{if } - \dfrac{m-1}{m}\eta < x < 0,
    \end{array} \right.
\end{align*}
and
\begin{align*}
    h''(x)= \left\{ \begin{array}{llrll}
    r(r-1)(\eta+x)^{r-2} -
    \dfrac{r(r-1)}{\eta^{2-r}}  - \dfrac{r(r-1)(r-2)}{\eta^{3-r}}  
     x - \dfrac{r(r-1)(r-2)(r-3)}{2\eta^{4-r}} x^2 \qquad  \text{if } x > 0,
    \vspace{0.15cm}\\
    r(r-1)(\eta+x)^{r-2}  - 
    \dfrac{r(r-1)}{\eta^{2-r}} - \left( \dfrac{m} {\eta}\right)^{3-r} r(r-1)(r-2) x
    \quad  \text{if } - \dfrac{m-1}{m}\eta < x < 0,
    \end{array} \right.
\end{align*}
and
\begin{align*}
    h'''(x)= \left\{ \begin{array}{llrll}
    r(r-1)(r-2)(\eta+x)^{r-3}   - \dfrac{r(r-1)(r-2)}{\eta^{3-r}}  
      - \dfrac{r(r-1)(r-2)(r-3)}{\eta^{4-r}} x \qquad  \text{if } x > 0,
    \vspace{0.15cm}\\
    r(r-1)(r-2)(\eta+x)^{r-3}   - \left( \dfrac{m} {\eta}\right)^{3-r} r(r-1)(r-2) 
    \quad  \text{if } - \dfrac{m-1}{m}\eta < x < 0.
    \end{array} \right.
\end{align*}
For all $x>0,$ we have $(\eta+x)^{4-r} > \eta^{4-r}.$ Thus,
\begin{align}
    h^{(4)}(x) &= r(r-1)(r-2) (r-3)(\eta+x)^{r-4} - \dfrac{r(r-1)(r-2)(r-3)}{\eta^{4-r}} \notag\\
    & = r(r-1)(r-2) (r-3) \left[\dfrac{1}{(\eta+x)^{4-r}} - \dfrac{1}{\eta^{4-r} } \right] \notag\\
    & = -r(1-r)(2-r)(3-r) \left[\dfrac{1}{(\eta+x)^{4-r}} - \dfrac{1}{\eta^{4-r} } \right] >0.
    \end{align}
  This implies that the third order derivative $h'''(x)$ is an increasing function. Hence, 
  \begin{align*}
    h'''(x) > h'''(0) = r(1-r)(2-r)\dfrac{1}{\eta^{3-r}}- \dfrac{r(1-r)(2-r)}{\eta^{3-r}} =0.
  \end{align*}
  This implies $h''(x)>h''(0)=0$, and, thus $h'(x)>h'(0)= 0$, and $h(x)>h(0)=0$. Thus, $f_{\eta,r}(x) \geq  f_{\eta,r}^{\mathsf{low}}(x)$ for all $x\geq 0$.

For all $- \dfrac{m-1}{m}\eta < x < 0,$ we have
\begin{align*}
  h'''(x) &=  r(r-1)(r-2)(\eta+x)^{r-3}   - \left( \dfrac{m} {\eta}\right)^{3-r} r(r-1)(r-2) \notag\\
  & = r(1-r)(2-r)\dfrac{1}{(\eta+x)^{3-r}} - \left( \dfrac{m} {\eta}\right)^{3-r} r(1-r)(2-r)  \notag\\ 
  & < r(1-r)(2-r)\left( \dfrac{m} {\eta}\right)^{3-r} - \left( \dfrac{m} {\eta}\right)^{3-r}r(1-r)(2-r)=0.
\end{align*}
Thus, $h''(x)>h''(0)=0$, and, thus, $h'(x)<h'(0)= 0,$ and, thus, $h(x)>0$. Thus, $f_{\eta,r}(x) > f_{\eta,r}^{\mathsf{low}}(x)$ for all $- \dfrac{m-1}{m}\eta < x < 0.$ 

  We continue by considering the function
  \begin{align*}
     u(x)& = g_{\rho}(x)- g_{\rho}^{\mathsf{low}}(x) \\
      & = \left\{ \begin{array}{llrll}
     \log(\rho+x) - \log \rho - \dfrac{x}{\rho} + \dfrac{x^2}{2\rho^2} -  \dfrac{x^3}{3\rho^3} +   \dfrac{x^4}{4\rho^4} & \text{if}& x \ge 0,
    \vspace{0.15cm}\\
    \log(\rho+x) - \log \rho - \dfrac{x}{\rho} + \dfrac{x^2}{2\rho^2} -   m^3 \dfrac{x^3}{3\rho^3}   & \text{if}& - \dfrac{m-1}{m}\rho < x < 0.  
    \end{array} \right.
  \end{align*}
  The case $x=0$ is trivial. 
  
We compute the derivatives of $u(x)$ and analyse their signs.
 \begin{align*}
     u'(x)= \left\{ \begin{array}{llrll}
     \dfrac{1}{\rho +x} - \dfrac{1}{\rho} + \dfrac{x}{\rho^2} -  \dfrac{x^2}{\rho^3} +   \dfrac{x^3}{\rho^4} & \text{if}& x  > 0,
    \vspace{0.15cm}\\
    \dfrac{1}{\rho +x} - \dfrac{1}{\rho} + \dfrac{x}{\rho^2} -   m^3 \dfrac{x^2}{\rho^3}   & \text{if}& - \dfrac{m-1}{m}\rho < x < 0,  
    \end{array} \right.
  \end{align*}
and
\begin{align*}
     u''(x)= \left\{ \begin{array}{llrll}
     \dfrac{-1}{(\rho +x)^2}  + \dfrac{1}{\rho^2} -  \dfrac{2x}{\rho^3} +   \dfrac{3x^2}{\rho^4} & \text{if}& x > 0,
    \vspace{0.15cm}\\
    \dfrac{-1}{(\rho +x)^2}   + \dfrac{1}{\rho^2} -   m^3 \dfrac{2x}{\rho^3}   & \text{if}& - \dfrac{m-1}{m}\rho < x < 0,  
    \end{array} \right.
  \end{align*}
and
\begin{align*}
     u'''(x)= \left\{ \begin{array}{llrll}
     \dfrac{2}{(\rho +x)^3}   -  \dfrac{2}{\rho^3} +   \dfrac{6x}{\rho^4} & \text{if}& x > 0,
    \vspace{0.15cm}\\
    \dfrac{2}{(\rho +x)^3}    -   m^3 \dfrac{2}{\rho^3}   & \text{if}& - \dfrac{m-1}{m}\rho < x < 0,  
    \end{array} \right.
  \end{align*}
For all $x>0,$ we have $u^{(4)}(x) = \dfrac{-6}{(\rho +x)^4} +\dfrac{6}{\rho^4}>0$, and, thus $u'''(x) > u'''(0)= 0 $, thus, $u''(x)> u''(0)=0$, thus, $u'(x)> u'(0) = 0$, thus, $u(x)> u(0)=0$. Thus, $g_{\rho}(x)\geq  g_{\rho}^{\mathsf{low}}(x)$ for all $x \geq 0$.

For all $- \dfrac{m-1}{m}\rho < x < 0,$ we get 
\begin{align*}
    u'''(x) =  \dfrac{2}{(\rho +x)^3}    -   m^3 \dfrac{2}{\rho^3}  
 < 2 \left(\dfrac{m}{\rho}\right)^3 -  m^3 \dfrac{2}{\rho^3} =0.  
\end{align*}
Thus, $u''(x)>u''(x) = 0,$ thus, $u'(x)<u'(0)= 0,$ and, thus, $u(x)>0$. Thus, $g_{\rho}(x)> g_{\rho}^{\mathsf{low}}(x)$ for all $- \dfrac{m-1}{m}\rho < x < 0.$ This finishes the proof.
\end{proof}

\subsection{Proof of Lemma \ref{Lem:ComputeYkj}.}
\label{Apd.Lem:ComputeYkj}
\begin{proof}
Using Lemma \ref{Lem:BasicProp_0},  we compute directly   
    \begin{align*}
    \mathbb{E} ({Y}_k^j) 
    &= \mathbb{E} \left[  \sum_{i=0}^j \begin{pmatrix}
    j \\
    i 
    \end{pmatrix} \big[ \alpha \Delta W_k \big]^i \left( \dfrac{\alpha^2}{2}(\Delta W_k)^2\right)^{j-i}   \right] = \mathbb{E} \left[  \sum_{i=0}^j \begin{pmatrix}
    j \\
    i 
    \end{pmatrix} \frac{\alpha^{2j-i}}{2^{j-i}} \Delta W_k^{2j-i}    \right]\\
    &= \mathbb{E} \left[  \sum_{i=0}^{\frac{1}{2}\jeven} \begin{pmatrix}
    j \\
    2i 
    \end{pmatrix}  \frac{\alpha^{2j-2i}}{2^{j-2i}} \Delta W_k^{2j-2i}    \right] =   \sum_{i=0}^{\frac{1}{2}\jeven} \begin{pmatrix}
    j \\
    2i 
    \end{pmatrix}  \frac{\alpha^{2j-2i}}{2^{j-2i}} [(2j-2i-1)!!] h^{j-i}    
\end{align*}
in which we note $j-i\ge j-\frac{1}{2}\jeven \ge 2$ for all $j\ge 3$. Here, we note that the above computations do not depend on $k$. 
Therefore, we imply that 
\begin{align*}
 \mathbb{E} \left({Y}_k -\dfrac{\alpha^2}{2} h \right)^m & =  \mathbb{E} \left[  \sum_{j=0}^m \begin{pmatrix}
    m \\
    j 
    \end{pmatrix} {Y}_k^j \left( -\dfrac{\alpha^2}{2} h \right)^{m-j}   \right]\\
    & = \sum_{j=0}^m \begin{pmatrix}
    m \\
    j 
    \end{pmatrix}   \Bigg[ \sum_{i=0}^{\frac{1}{2}\jeven} \begin{pmatrix}
    j \\
    2i 
    \end{pmatrix}  \frac{\alpha^{2j-2i}}{2^{j-2i}} [(2j-2i-1)!!] h^{j-i}  \Bigg ] \left( -\dfrac{\alpha^2}{2} h \right)^{m-j}\\ 
    & = \sum_{j=0}^m \begin{pmatrix}
    m \\
    j 
    \end{pmatrix}   \sum_{i=0}^{\frac{1}{2}\jeven} \begin{pmatrix}
    j \\
    2i 
    \end{pmatrix} (-1)^{m-j} \frac{\alpha^{2m-2i}}{2^{m-2i}} [(2j-2i-1)!!] h^{m-i} \\
    & =  \left\{
    \begin{array}{llll}
    0     & \text{if} \quad m=1, \\
    \alpha^2h + \dfrac{1}{2}\alpha^4h^2  & \text{if} \quad m=2, \\
    O(h^2)   & \text{if} \quad m \ge 3.
    \end{array}
    \right.
\end{align*}
This consequently shows the expectations of $\mathbb{E} (N_k^l)$ for $l=1,2,3$. 
\end{proof}

\subsection*{Acknowledgement}
The authors sincerely thank Dr. Bao-Ngoc Tran for his valuable and fruitful discussions.
This research was initiated during the visit of Hue Vu to University of Graz, which was supported by the ASEA-UNINET project number ASEA 2023-2024/Uni Graz/6. The hospitality of the university is acknowledged.

\section*{Declarations}


\subsection*{Competing Interests} The authors declare no competing interests.

\subsection*{Ethics approval} This study did not require ethics approval.

\subsection*{Availability of Data and Materials} No datasets were generated or analysed during the current study.

\subsection*{Author Contributions}
All authors contributed equally to this work. All authors have read and approved the final version of the manuscript.

\newcommand{\etalchar}[1]{$^{#1}$}

\end{document}